\RequirePackage[2020-02-02]{latexrelease}
\documentclass[nodoi]{arxivpaper}
\providecommand\DeclareCommandCopy[2]{\let#1#2}

\usepackage{amsmath,amssymb,amsthm}
\usepackage{mathtools} 
\usepackage{comment}
\usepackage{enumerate}
\usepackage{tikz}
\usepackage{float}
\usetikzlibrary{arrows.meta, positioning, shapes.misc}
\usepackage{graphicx} 
\ifpdf
\usepackage{epstopdf}
\usepackage{url}
\fi
\providecommand{\href}[2]{#2}

\usetikzlibrary{calc}
\usetikzlibrary{arrows.meta}

\makeatletter
\let\ltx@label\label
\makeatother

\long\def\ignore#1{}

\newcommand{\ZRC}{Z_{\mathrm{RC}}}
\newcommand{\THG}{T_{\mathrm{HG}}}

\newcommand{\ZP}{Z_{\mathrm{P}}}
\newcommand{\ZGr}{Z_{\mathrm{Grimmett}}}

\makeatletter

\def\tbl#1#2#3{\global\setbox\tabcapbox\vbox{#1}\addtocounter{table}{-1}\global\setbox\tabbox\hbox{#2}\global\setbox\tabnotebox\vbox{#3}\noindent\vbox{\tablemove\textwidth\advance\tablemove-\wd\tabbox\divide\tablemove2\vbox{\hsize\wd\tabbox#1#2#3}}}
\makeatother

\begin{document}
\label{firstpage}

\lefttitle{Berrekkal, Ellis-Monaghan and Moody}
\righttitle{A Hypergraph Tutte Polynomial}
\jnlPage{\pageref{firstpage}}{\pageref{lastpage}}
\jnlDoiYr{2026}
\doival{}

\title{A Hypergraph Tutte Polynomial}

\renewcommand{\email}[1]{#1}

\begin{authgrp}
\centering
Khallil Berrekkal\textsuperscript{1}\quad Joanna A. Ellis-Monaghan\textsuperscript{2}\quad Merijn Moody\textsuperscript{3}\par
\end{authgrp}

\begin{abstract}
  We introduce a Tutte polynomial for hypergraphs,  $\THG$, together with $T_k$, a related Tutte polynomial for $k$-polymatroids. 
  Both invariants admit deletion--contraction recursions that remain within their respective classes, and they are linked by the fact that $\THG$ specializes to $T_k$ on the associated polymatroid of any $(k+1)$-uniform hypergraph. 
  We show that $\THG$ satisfies several desirable Tutte type properties, including multiplicativity and duality, while $T_k$ further admits a universality theorem, as well as a convolution product formula. 
  In the uniform hypergraph setting, these latter results specialize back to $\THG$. 
  We also relate $\THG$ to hypergraph extensions of the Potts and random cluster models. In particular, we study degree dependent random cluster and Potts partition functions, as well as Grimmett's many body Potts model, and compare their relationship with $\THG$ in both the general and uniform settings.
  Finally, we compare $\THG$ with the polymatroid Tutte polynomial $\mathcal{T}_P$ of Bernardi, K\'alm\'an, and Postnikov, showing that the two are incomparable in distinguishing power. As a consequence, we answer negatively a question raised by these authors by proving that the characteristic polynomial is not, in general, a specialization of $\mathcal{T}_P$.
\end{abstract}

\mscxx
\begin{msc}
{05C31; 05B35; 05C65; 82B20}
\end{msc}

\begin{keywords}
Hypergraph Tutte polynomial; polymatroids; deletion--contraction; hypergraph Potts model; random cluster model.
\end{keywords}

\maketitle
\renewcommand{\thefootnote}{\arabic{footnote}}
\footnotetext[1]{Korteweg-de Vries Institute for Mathematics, University of Amsterdam, Science Park 105-107, 1098 XG Amsterdam, The Netherlands. Email: k.berrekkal@uva.nl. Supported by the Netherlands Organisation for Scientific Research (NWO) under grant VI.Vidi.193.068.}
\footnotetext[2]{Korteweg-de Vries Institute for Mathematics, University of Amsterdam, Science Park 105-107, 1098 XG Amsterdam, The Netherlands. Email: jellismonaghan@gmail.com.}
\footnotetext[3]{Korteweg-de Vries Institute for Mathematics, University of Amsterdam, 1090 GE Amsterdam, The Netherlands; Institute of Physics, University of Amsterdam, 1090 GE Amsterdam, The Netherlands; Dutch Institute for Emergent Phenomena, 1090 GL Amsterdam, The Netherlands. Email: merijnmoody@gmail.com. Supported by the Dutch Institute for Emergent Phenomena (DIEP) cluster at the University of Amsterdam via the DIEP programme Foundations and Applications of Emergence (FAEME).}

\section{Introduction}\label{sec: Introduction}

In this paper, we introduce a new extension of the Tutte polynomial to hypergraphs and $k$-polymatroids, and we prove that many familiar Tutte properties extend to this setting, including a deletion--contraction recursion.
While there are several well-developed Tutte polynomial extensions to polymatroids, see for instance \cite{bernardi2022universal, chavez2021new}, for hypergraphs there does not yet seem to be a Tutte polynomial generalization that admits a deletion--contraction recursion while remaining within the class of hypergraphs.  Related investigations are already underway however, e.g. \cite{Moody2026}.

One of the significant challenges in extending the Tutte polynomial to hypergraphs is determining and choosing appropriate definitions for deletion, contraction, rank, and other parameters analogous to those for graphs.
Each such choice can give rise to different generalizations of the Tutte polynomial, with different advantages and disadvantages.
This phenomenon mirrors what arose in extending the Tutte polynomial to topological settings such as embedded and ribbon graphs, where several different generalizations were developed over the years, for example in work of Las Vergnas \cite{las1978activites, las1980tutte}, Bollob\'as and Riordan \cite{bollobas2001polynomial, bollobas2002polynomial}, Krushkal \cite{krushkal2011graphs}, Butler \cite{butler2018quasi}, and Goodall et al. \cite{goodall2018tutte}.
These were eventually resolved with a unifying theory provided by Moffatt and Huggett in \cite{huggett2020types} (see also~\cite{krajewski2018hopf}).
For hypergraphs, the theory seems still to be at the stage of discovering the possibilities.
The present work contributes to this with a new extension to hypergraphs and polymatroids based on a set of natural choices for hypergraph and polymatroid parameters, and a detailed analysis of the comparability of existing models.
The polynomial $\THG$ has the advantage that it satisfies several desirable Tutte-type properties directly for hypergraphs, including multiplicativity and duality.
Moreover, $\THG$ is obtained in the uniform hypergraph setting from the more general polynomial $T_k$, for which we prove a universality theorem and a convolution product formula.
The present work thus forwards the ultimate goal of a unifying theory of Tutte polynomials for hypergraphs.

We recall the definition of the Tutte polynomial for a graph $G=(V,E)$:
\[
T(G;x,y)= \sum_{A \subseteq E}(x-1)^{\kappa(A)-\kappa(G)}(y-1)^{\kappa(A)+|A|-|V|},
\]
where $\kappa(A)$ denotes the number of components of the spanning subgraph induced by $A \subseteq E$.
As a graph invariant, the Tutte polynomial admits many specializations with clean combinatorial interpretations.
It also satisfies a deletion--contraction recursion, a convolution formula, and  a duality formula for planar graphs.
Moreover, the Tutte polynomial has a universality property, in the sense that every multiplicative graph parameter satisfying a suitable deletion--contraction recursion is, up to a change of variables, obtained by evaluating the Tutte polynomial.

\paragraph{Beyond graphs.}
The Tutte polynomial is also a matroid invariant, and this perspective is important for our later extension to polymatroids.
Indeed, every graph $G$ gives rise to a matroid on $E(G)$, its graphic matroid, and $T(G;x,y)$ depends only on this underlying matroid structure.
In this broader setting we write $T(M;x,y)$ for an arbitrary matroid $M$, and the familiar deletion--contraction recursion is then simply the matroid deletion--contraction.

This paper considers two settings, namely hypergraphs and $k$-polymatroids, both of which arise naturally as generalizations of this classical picture.
Polymatroids form a natural extension of matroids.
While a graph already comes with a matroid rank function on its edge set, the same approach applied to hypergraphs produces a polymatroid rather than a matroid.
Polymatroids still support the same basic minor operations as matroids, so we can apply deletion and contraction in a familiar way~\cite{oxley1993characterization, whittle1992duality}.
Polymatroids appear in several areas, including combinatorial optimization and submodular theory, and they provide a natural ambient setting to extend the Tutte polynomial and minor operations beyond matroids~\cite{fujishige2005submodular, Savitsky14, schrijver2003combinatorial}.

Hypergraphs, on the other hand, generalize graphs in a different direction.
From a combinatorial viewpoint, a graph only records pair interactions, while a hypergraph allows the hyperedges to record higher order interactions.
This added flexibility makes hypergraphs useful whenever we want to model higher-order structure rather than reduce everything to pairwise relations.
For this reason, hypergraphs appear naturally in several mathematical and applied settings and in particular in models where interactions need not be pairwise~\cite{battiston2021physics, robiglio2025higher, son2024phase}; see also~\cite{spinmod, Mastromatteo2013} for earlier high-order spin-model perspectives.

With these two generalizations in mind, we now describe the extensions introduced in this paper.

\paragraph{A hypergraph Tutte polynomial.}
In Section~\ref{sec: hypertutte}, we define a Tutte polynomial $\THG$ for hypergraphs.
This polynomial takes inspiration from the coarse Tutte polynomial for hypermaps introduced in \cite{ellis2024coarse}.
The coarse hypermap Tutte polynomial is an invariant of hypergraph embeddings on surfaces, and our $\THG$ may be viewed as a non-topological analogue of it.
An important feature carried over from \cite{ellis2024coarse} is that we record edge degrees, allowing hyperedges to contain the same vertex multiple times, i.e., with multiplicity.
This generalizes loops from graphs to hypergraphs and, crucially, makes the deletion--contraction recursion for $\THG$ tractable, with the degree data built into the recursion itself.
We make this precise in Section~\ref{sec: hypertutte}, where we introduce our hypergraph minor operations and prove the deletion--contraction recurrence for $\THG$ on general hypergraphs.

Section~\ref{sec: polymatroids} develops the $k$-polymatroid Tutte polynomial $T_k$.
We prove deletion--contraction, multiplicativity, duality, specializations to the characteristic polynomial, a universality theorem, and a convolution product formula.
Section~\ref{sec: Hypergraphs} then returns to hypergraphs.
We describe the associated polymatroid construction and prove that our hypergraph minors correspond to polymatroid minors.
We also specialize to $(k+1)$-uniform hypergraphs and derive the corresponding universality theorem and convolution identities in the hypergraph setting, as well as the planar duality statement inherited from \cite{ellis2024coarse}.

Since the associated polymatroids do not capture hyperedge degrees, $\THG$ does not extend to polymatroids in general.
However, within our framework we encode $(k+1)$-uniform hypergraphs as $k$-polymatroids, which yields a natural extension of $\THG$ on this restricted class.
We denote this specialized extension by $T_k$.
Thus, besides the hypergraph polynomial $\THG$, our construction also yields the extension $T_k$ on a restricted polymatroid class that still retains the relevant hypergraph structure.
This also reflects the fact that the hypergraph and polymatroid viewpoints are closely related, but not identical, in our setting.

\paragraph{Potts model partition functions.}
A further motivation for generalizing the Tutte polynomial to hypergraphs comes from statistical mechanics, where the classical Tutte polynomial plays a well-known organizing role through its specializations to partition functions such as the Potts model and the random cluster model.
The $q$-state Potts model partition function of a graph $G$ is the normalization factor for the Boltzmann probability distribution, and can be written as
$Z(G)=\sum_{\sigma}\exp(-\beta\,h(\sigma))$,
where the sum ranges over all spin states $\sigma$ on a graph $G$ with $q$ possible spins.
Here $\beta$ is the inverse temperature, and the Hamiltonian $h$ measures the energy of the state; its precise form depends on the application.
In the absence of an external magnetic field, and with constant interaction energies between adjacent vertices, the Hamiltonian becomes
\[
h(\sigma)=-J\sum_{\{i,j\}\in E(G)}\delta(\sigma_i,\sigma_j),
\]
where $\delta$ denotes the Kronecker delta.
In this case, the classical Tutte polynomial evaluates to the Potts partition function:
\[
Z(G;q,w)  =q^{\kappa(G)}(w-1)^{|V(G)|-\kappa(G)}T\left(G; \frac{q+w-1}{w-1}, w \right),
\]
where $w=e^{J\beta}$.
(See \cite{FK72} for the nascent stages of this theory and later exposition in \cite{Bax82, BE-MPS10, Big93, Bol98, Loe10, Tut84, Wel93, WM00}.)
This relationship has resulted in a remarkable synergy between combinatorics and statistical mechanics, particularly in computational complexity and the study of zeros of these polynomials; see \cite{BE-MPS10, GJ07, Roy09, Sok05, WM00} for surveys on these results.

Closely related to the Potts model is the random cluster model of Fortuin and Kasteleyn, which rewrites the partition function as a sum over edge subsets rather than spin configurations.
In the graph setting, the Tutte polynomial again encodes this partition function~\cite{FK72, grimmett2004random}.

For the purposes of this paper, the key point is that interactions in statistical physics need not be pairwise.
Several models feature many-body couplings, where an energy term depends on the joint state of three or more spins at once~\cite{battiston2021physics, robiglio2025higher, son2024phase};  see also~\cite{Wegner1971, Mastromatteo2013, spinmod} for earlier high-order Ising and spin-model work.
For a complementary framework connecting such higher-order statistical physics models with hypergraph theory, see~\cite{Moody2026}.

From this perspective, hypergraphs are not merely a formal generalization of graphs, but a natural combinatorial framework for higher-order statistical mechanical models.
In particular, Geoffrey Grimmett gives a hypergraph formulation of both the Potts model and the random cluster model \cite{grimmett1994potts}.

This statistical-mechanical perspective provides a direct motivation for Tutte polynomial-like invariants of hypergraphs: we seek a polynomial that plays an analogous organizing role to the graph Tutte polynomial, while retaining a robust deletion--contraction theory and interfacing naturally with hypergraph Potts and random cluster partition functions.
The polynomials $\THG$ and $T_k$ are introduced in this spirit.

In Section~\ref{sec: potts}, we develop this connection with hypergraph extensions of the Potts and random cluster models.
We show that $\THG$ is equivalent to a degree-dependent hypergraph random cluster partition function $\ZRC$, via explicit change-of-variable formulas.
Thus $\THG$ plays, in this hypergraph setting, the same organizing role for $\ZRC$ that the classical Tutte polynomial plays for the graph random cluster model.
We also compare $\ZRC$ with Grimmett's many-body Potts model and a degree-dependent hypergraph Potts partition function, showing that these three hypergraph partition functions are genuinely distinct in general, and that among them only $\ZRC$ is equivalent to $\THG$.
In the uniform setting, however, they become equivalent.

\paragraph{Comparison with existing polymatroid Tutte polynomials.}
Besides relating $\THG$ to statistical mechanical partition functions, we also compare it with the polymatroid Tutte polynomial $\mathcal{T}_P$ of Bernardi, Kálmán, and Postnikov \cite{bernardi2022universal}.
This comparison is particularly natural in the uniform setting, where $(k+1)$-uniform hypergraphs give rise to associated $k$-polymatroids and $\THG$ specializes to $T_k$.
Nevertheless, in Section~\ref{sec: comparison} we show that $\THG$ and $\mathcal{T}_P$ are incomparable in distinguishing power, even on uniform examples.
As a consequence, in Corollary~\ref{cor:char-not-specialization}, we answer negatively a question raised by Bernardi, Kálmán, and Postnikov in Section 18.3 of \cite{bernardi2022universal} by proving that the characteristic polynomial cannot, in general, be obtained as a specialization of $\mathcal{T}_P$.
We conclude in Section~\ref{sec: conclusion} with further directions and open problems.

\section{The Hypergraph Tutte Polynomial $T_{\mathrm{HG}}$}\label{sec: hypertutte}
In this section we introduce our hypergraph conventions, the corresponding minor operations, and the hypergraph Tutte polynomial $\THG$.
We then prove the deletion--contraction recurrence for $\THG$ on general hypergraphs and record its basic properties.

\begin{definition}[Hypergraph]\label{def: Hypergraph}
A \emph{hypergraph} is a pair $H=(V,E)$, where $V$ is a finite set, whose elements are called \emph{hypervertices}, and $E$ is a multi-collection of multi-subsets of $V$, whose elements are called \emph{hyperedges}.
\end{definition}
Multi-edges are common in the literature; here we additionally treat $E$ itself as a multiset of multisets, analogous to allowing loops in graphs, meaning that the same vertex may appear multiple times within the same edge.
We write $v(H)=|V|$,  and $e(H)=|E|$ denotes the number of hyperedges with multiplicity, i.e. counting parallel hyperedges. 
For $e\in E$, let $d(e)$ be its size counting multiplicity, and $v(e)$ be the number of \emph{distinct} vertices in $e$.
We say that $H$ is \emph{$k$-uniform} if $d(e)=k$ for all $e \in E$.
In this paper, we always assume that $d(e)>1$ for any hyperedge $e$.

Given a hypergraph $H=(V,E)$, the \emph{associated bipartite incidence graph} $\mathrm{Bip}(H)$ is the bipartite graph whose vertex classes are $V$ and $E$. 
For $v\in V$ and $e\in E$, the vertices $v$ and $e$ are connected by exactly $m_e(v)$ parallel edges, where $m_e(v)$ denotes the multiplicity of $v$ in the hyperedge $e$.
The number of connected components of $H$ is denoted by $\kappa(H)$ and is defined by $\kappa(H)\coloneqq \kappa(\mathrm{Bip}(H))$.
Moreover, we call $H$ \emph{planar} if and only if $\mathrm{Bip}(H)$ is planar.

For $e\in E$, we define
\[
\Delta(e)\coloneqq \kappa(H\setminus e)-\kappa(H),
\qquad
m(e)\coloneqq d(e)-v(e).
\]
Given a subset of hyperedges $A\subset E$, we denote the \emph{subhypergraph} as $H|_A=(V,A)$.
For a subset $A\subseteq E$, we write  $d(A)=\sum_{e\in A} d(e)$, and $\kappa(A)$  for $\kappa(H|_A)$.

We now introduce an extension of the Tutte polynomial to hypergraphs that we will denote as $\THG$, together with its deletion and contraction property.

\begin{definition}[Hypergraph Tutte polynomial]\label{def: HyperTutte}
For a hypergraph $H$, we define its Tutte polynomial by
    \[
    T_{\mathrm{HG}}(H ; x,y) =  \sum_{A \subseteq E(H)} (x-1)^{\kappa(A)- \kappa(H) } (y-1)^{d(A) - |A| - v(H)+\kappa(A)}.
    \]
\end{definition}
This extension is inspired by the coarse Tutte polynomial for hypermaps, as introduced in  \cite{ellis2024coarse}, in the sense that our hypergraph Tutte polynomial can be seen as a non-topological analogue of the coarse Tutte polynomial for hypermaps. 
The main difference between the classical Tutte polynomial and the hypergraph Tutte polynomial is that the latter incorporates the degrees of the edges.
For hypergraphs, the quantity $\kappa(A) +|A| -v(H)$  may be negative, whereas $d(A) - |A| - v(H)+\kappa(A) \ge 0$, ensuring that $\THG$ is in fact a polynomial.
For graphs, $d(A)= 2|A|$, and therefore $\THG$ agrees with the Tutte polynomial on graphs.
A key property of the classical Tutte polynomial is its deletion-contraction recursion. 
We define deletion and contraction operations of hyperedges for  hypergraphs, generalizing the classical graph-theoretic operations.
This requires extra care due to the presence of multi-edges and multiset edge structures in hypergraphs.

\begin{definition}[Hypergraph deletion and contraction]  \label{def: del/con}
Let $H = (V, E)$ be a hypergraph, and let $e \in E$ be a hyperedge. 

\begin{enumerate}
    \item
The \emph{deletion} of $e$ from $H$, denoted $H \setminus e$, is the hypergraph
\[
H \setminus e := (V, E \setminus \{e\}),
\]
with the vertex set $V$ unchanged.

    \item
The \emph{contraction} $H/e$ is the hypergraph $(V',E')$ defined by
\[
V' \;=\;(V\setminus e)\;\cup\;\{v_e\},
\qquad
E' \;=\;\bigl\{\,f':\,f\in E\setminus\{e\}\bigr\},
\]
where each $f'$ is the multiset obtained from $f$ by replacing every occurrence of a vertex in $e$ with the new vertex $v_e\,$.  
In other words, if $f \neq e$ contains $k$ elements of $e$ (with multiplicity), then 
\[
f'=\bigl(f\setminus e\bigr)\;\cup\;\bigl\{\underbrace{v_e,\dots,v_e}_{k\text{ times}}\bigr\}.
\]
\end{enumerate}

\end{definition}
Informally, contracting $e$ is obtained by deleting $e$, identifying all vertices of $e$ into a single vertex, and updating the incidences of the remaining edges so that their degrees are preserved.
To the best of our knowledge, this formal version of hyperedge contraction preserving edge degrees has not previously appeared in the literature.
Another definition of hyperedge contraction appears in \cite{steiner2022coloring}, but there the degrees are not preserved, and parallel edges are not retained.

This degree preservation enables us to establish a full deletion--contraction recurrence for $\THG$.
In particular, uniformity is preserved as well, which allows us to study properties of $\THG$ within the class of $k$-uniform hypergraphs that do not necessarily hold for general hypergraphs.

We will now list some properties of $\THG$, starting with the deletion--contraction recurrence.

\begin{theorem}
\label{thm:dc-THG}
Let $H=(V,E)$ be a hypergraph and let $e\in E$, where deletion and contraction are as in Definition~\ref{def: del/con}.
Then
\[
      T_{\mathrm{HG}}(H ; x,y) = (x-1)^{
      \Delta(e)
      }T_{\mathrm{HG}}(H\setminus{e};x,y) + (y-1)^{m(e)}T_{\mathrm{HG}}(H/e;x,y).
\]
\end{theorem}

\begin{proof}
We split the subset expansion in Definition~\ref{def: HyperTutte} obtaining 
\begin{align*}
T_{\mathrm{HG}}(H;x,y)
=&\sum_{\substack{A\subseteq E\\ e\notin A}}
(x-1)^{\,\kappa(A)-\kappa(H)}\,
(y-1)^{\,d(A)-|A|-v(H)+\kappa(A)}
\\ &+\sum_{\substack{A\subseteq E\\ e\in A}}
(x-1)^{\,\kappa(A)-\kappa(H)}\,
(y-1)^{\,d(A)-|A|-v(H)+\kappa(A)}.
\end{align*}

The first sum is simply $\sum_{A \subseteq E(H\setminus{e})}
(x-1)^{\,\kappa(A)-\kappa(H)}\,
(y-1)^{\,d(A)-|A|-v(H)+\kappa(A)}$.
The number of vertices in $H\setminus{e}$ remains unchanged, and we write $\kappa(H)=\kappa(H) -\kappa(H\setminus{e}) + \kappa(H\setminus{e})$, thus the first term becomes
\[
(x-1)^{\kappa(H\setminus{e})  - \kappa(H)} \sum_{A \subseteq E(H\setminus{e})}
(x-1)^{\,\kappa(A)-\kappa(H\setminus{e})}\,
(y-1)^{\,d(A)-|A|-v(H\setminus{e})+\kappa(A)}.
\]

For the second sum, we set $A'=A\setminus\{e\}$, giving a bijection
$\{A\subseteq E: e\in A\}\leftrightarrow \{A'\subseteq E\setminus\{e\}\}$.
From the definition of contraction, it follows that $\kappa\left((H/e)\big|_{A'}\right)=\kappa\left(H\big|_{A}\right)$, $\kappa(H/e)=\kappa(H)$, $d(A)=d(A')+d(e)$, $|A|=|A'|+1$, and $v(H/e)=v(H)-v(e)+1$.
Thus, the second term can be rewritten as
\begin{align*}
    &\sum_{\substack{A\subseteq E\\ e\in A}}
(x-1)^{\,\kappa(A)-\kappa(H)}\,
(y-1)^{\,d(A)-|A|-v(H)+\kappa(A)}\\
=& \sum_{\substack{A\subseteq E\\ e\in A}}
(x-1)^{\,\kappa(A')-\kappa(H/e)}\,
(y-1)^{\,d(A')+d(e)-|A'|-1-v(H/e)-v(e)+1+\kappa(A')}\\
=&  (y-1)^{d(e)-v(e)}\sum_{\substack{A'\subseteq E(H/e)}}
(x-1)^{\,\kappa(A')-\kappa(H/e)}\,
(y-1)^{\,d(A')-|A'|-v(H/e)+\kappa(A')}\\
=&(y-1)^{\,m(e)}\,T_{\mathrm{HG}}(H/e;x,y),
\end{align*}
where $m(e)=d(e)-v(e)$.

After noticing that $\kappa(H \setminus{e}) - \kappa(H)=\Delta(e)$, combining the two parts gives
\[
T_{\mathrm{HG}}(H;x,y)
=(x-1)^{\,\Delta(e)}\,T_{\mathrm{HG}}(H\setminus e;x,y)
\;+\;
(y-1)^{\,m(e)}\,T_{\mathrm{HG}}(H/e;x,y),
\]
as claimed.
\end{proof}

We note that this deletion--contraction formula generalizes the graph case.
For graphs, both $\Delta(e)$ and $m(e)$ take values in $\{0,1\}$, and they cannot simultaneously be equal to $1$ for the same edge.

\begin{proposition}
    Let $H_1,H_2$ be vertex- and edge-disjoint hypergraphs.   The hypergraph Tutte polynomial $\THG$ satisfies the following properties.

    \begin{enumerate}[(i)]
    \item  
    The hypergraph Tutte polynomial $\THG$  is multiplicative over disjoint unions, i.e.
    \[
        \THG(H_1 \cup H_2)=\THG(H_1)\,\THG(H_2).
    \]

    \item 
    
    Let $H_1 * H_2$ be the vertex join of $H_1$ and $H_2$ obtained by identifying a vertex $v_1\in V(H_1)$ with a vertex $v_2\in V(H_2)$ into a single vertex $v$ (and leaving all edges otherwise unchanged). Then
    \[
        \THG(H_1 * H_2)=\THG(H_1)\,\THG(H_2).
    \]
\end{enumerate}
\end{proposition}
\begin{proof}
    The proof of (i) is very standard, so we spare the reader the details.
    For part (ii), let $H= H_1 * H_2$. 
    We denote by $E_1$ and $E_2$  the edge sets in $H$  coming from $H_1$ and $H_2$ respectively, with $E_1 \cap E_2 = \emptyset$.
    We have
    \begin{align*}
        \THG(H;x,y) &= \sum_{A \subseteq E(H)}(x-1)^{\kappa_H(A)-\kappa_H(H)}(y-1)^{ d(A)-|A| - v(H)+\kappa_H(A)}\\
        &=\sum_{A_1 \subseteq E_1}\sum_{A_2 \subseteq E_2} (x-1)^{\kappa_H(A_1 \cup A_2)-\kappa_H(H)}(y-1)^{ d(A_1 \cup A_2)-|A_1 \cup A_2| - v(H)+\kappa_H(A_1 \cup A_2)}.
    \end{align*}

    For $A_1,A_2$ it holds that $\kappa_H(A_1 \cup A_2)= \kappa_{H_1}(A_1) + \kappa_{H_2}(A_2)-1$, where the $-1$ comes from the fact that the component of $v_1$ in $(V(H_1),A_1)$  and the component of $v_2$ in $(V(H_2),A_2)$ get merged into one component in $(V(H),A_1 \cup A_2)$.
    Similarly $\kappa_H(H)= \kappa_{H_1}(H_1) + \kappa_{H_2}(H_2)-1$.
    We also observe that $v(H) = v(H_1) + v(H_2) -1$, where the $-1$ comes from double counting $v$.  Hence

    \[
        \kappa_H(A_1\cup A_2)-\kappa_H(H)
        =\big(\kappa_{H_1}(A_1)-\kappa_{H_1}(H_1)\big)+\big(\kappa_{H_2}(A_2)-\kappa_{H_2}(H_2)\big),
    \]
    and
    \[
    v(H) - \kappa_H(A_1 \cup A_2) =\big( v(H_1)  - \kappa_{H_1}(A_1)\big) + \big( v(H_2)  - \kappa_{H_2}(A_2)\big).
    \]
    Moreover, $d(A_1\cup A_2)=d_{H_1}(A_1)+d_{H_2}(A_2)$ and $|A_1 \cup A_2| = |A_1| + |A_2|$.
 
    Therefore the summand factors as a product of a term depending only on $A_1$ and a term depending only on $A_2$. The double sum then splits, giving 
    \[
        \THG(H;x,y)=\THG(H_1;x,y)\,\THG(H_2;x,y),
    \]
    which proves (ii).

\end{proof}

The following result follows directly from Proposition 5 in \cite{ellis2024coarse}, as the coarse hypermap Tutte polynomial and the hypergraph Tutte polynomial $\THG$ coincide for planar hypergraphs.
\begin{proposition}
    Let $H$ be a planar hypergraph, and let $H^*$ denote its vertex-face dual as defined in \cite{ellis2024coarse}. Then
    \[
    \THG(H;x,y) = \THG(H^*;y,x).
    \]
\end{proposition}

We have observed that several fundamental properties of the Tutte polynomial still hold for the hypergraph Tutte polynomial $\THG$ of arbitrary hypergraphs (e.g. deletion--contraction, multiplicativity and duality).
Restricting to 
$(k+1)$-uniform hypergraphs allows us to extend this list of properties, as we will see in Section~\ref{sec: Hypergraphs}.
In the next section, we turn to polymatroids, which provide a broader rank function setting containing the structures arising from $(k+1)$-uniform hypergraphs.

\section{A $k$-Polymatroid Tutte Polynomial}\label{sec: polymatroids}

In this section, we first work in the rank function setting of $k$-polymatroids and introduce the polynomial $T_k$, which serves as the polymatroid analogue of $\THG$ for the uniform part of our hypergraph theory.
This provides a natural intermediate framework, since $\THG$ itself does not extend to arbitrary polymatroids (the hyperedge-degree information is lost), whereas $(k+1)$-uniform hypergraphs correspond naturally to $k$-polymatroids in our setup.
By establishing in this section the main properties for $T_k$, in particular universality, a convolution formula, and other Tutte type properties, we can then return in the next section to hypergraphs and specialize these results to the uniform hypergraph setting via the associated polymatroid construction.

\subsection{Polymatroids}

We start by introducing \emph{polymatroids}, a natural generalization of \emph{matroids}. 
The definitions in this subsection follow those in \cite{chavez2021new, oxley1993characterization} closely.
\begin{definition}
    Let $E$ be a finite set, and let $r$ be a function $r:2^E \to \mathbb{Z} $. 
We call the pair $P=(E,r)$ a \emph{polymatroid}, if
\begin{enumerate}[i.]
    \item $r(\emptyset)=0$,
    \item $r$ is increasing, i.e. $r(A) \le r(B)$ if $A \subseteq B$, and
    \item $r$ is submodular, i.e. $r(A ) + r(B) \geq r(A 
\cup B) + r(A\cap B)$ for all $A,B \subseteq E$.
\end{enumerate}
Here, $E$ is known as the \emph{groundset} and $r$ is known as the \emph{rank} function. 
If $r(\{e\}) \leq k$ for all $e\in E$, then $P$ is called a \emph{$k$-polymatroid}.
\end{definition}

If any ambiguity may arise regarding a polymatroid $P$ under consideration, we distinguish its rank function by a subscript and write $r_P$.
Furthermore, we may write $r(P)$ instead of $r(E)$, if it is clear from the context, and we may denote its groundset as $E(P)$, and write $|P|$ for $|E(P)|$.
Moreover, if $e \in E$, then we write $r(e)$ for $r(\{e\})$.

There is no universally accepted definition of polymatroids in the literature.
Some authors allow the rank function $r$ to take values in $\mathbb{R}$~\cite{bonin2023natural}, while other authors do not require $r$ to be increasing~\cite{bernardi2022universal, guan2023deletion, kalman2013version}.

Our definition of a $k$-polymatroid specifically follows \cite{bonin2023natural, chavez2021new, oxley1993characterization, whittle1992duality}.
A $1$-polymatroid in this case is also known as a \emph{matroid}.
From this point onward, we fix an integer $k \ge 1$.
Unless stated otherwise, all polymatroids in this section are $k$-polymatroids.

Polymatroids have basic deletion and contraction operations.  
Given a subset $A\subseteq E$, the deletion of $A$ from $P$ is given by $P\setminus{A}=(E \setminus{A}, r_{P\setminus{A}})$, with $r_{P\setminus{A}}(X)=r(X)$ for all $X \subseteq E\setminus{A}$, and contraction of $A$ is given by $P/A = (E \setminus{A}, r_{P/A})$, with $r_{P/A}(X) = r(X \cup A) -r(A)$ for all $X \subseteq E\setminus{A}$. 
If $P$ is a $k$-polymatroid, then both $P/A$ and $P\setminus{A}$ are again $k$-polymatroids.
The restriction $P|_A$ denotes the polymatroid obtained from deleting $E\setminus{A}$. 
A \emph{minor} of $P$ is any polymatroid of the form $(P/A)\setminus{B}$ (equivalently $(P\setminus{B})/A$) for disjoint subsets $A,B \subseteq E$.
If $e\in E$, we simply write $P\setminus{e}$ and $P/e$, shorthand for $P\setminus\{e\}$ and $P/\{e\}$ respectively.

Given a $k$-polymatroid $P=(E,r)$, its \emph{$k$-dual} is the polymatroid $P^*=(E,r^*)$, with the same groundset $E$, and  a new (dual) rank function, $r^*$, defined by $r^* (A) = k|A| +r(E \setminus{A}) -r (E)$, for $A \subseteq E$. 
It is not hard to verify that $P^{*}$ is again a $k$-polymatroid, and that the map sending $P$ to $P^*$ is an involution, i.e. $( P^*) ^* =P$ \cite{chavez2021new, whittle1992duality}. 
Moreover, the dual interchanges deletion and contraction, i.e.  $P^*/e = (P\setminus{e})^*$ and, equivalently, $P^*\setminus{e} = (P/e)^*$.
Since our focus is $k$-polymatroids, we will simply refer to the $k$-dual of a $k$-polymatroid as its \emph{dual} if there is no ambiguity. 

Finally, given two polymatroids $P_1=(E_1,r_1)$ and $P_2=(E_2,r_2)$, with $E_1 \cap E_2 = \emptyset$, their \emph{direct sum} is $P_1 \oplus P_2 = (E_1 \cup E_2 , r_1 \oplus r_2)$, where $r_1 \oplus r_2(X) = r_1(X \cap E_1) + r_2(X \cap E_2)$, for $X \subseteq E_1 \cup E_2$.

\subsection{Definition and Properties}

We now state our version of the Tutte polynomial for $k$-polymatroids.

\begin{definition}[$k$-polymatroidal Tutte polynomial]\label{def: k-polyTutte}
For a $k$-polymatroid $P=(E,r)$ define 
\[
T_{k}(P;x,y)\coloneqq
\sum_{A\subseteq E}
(x-1)^{\,r(E)-r(A)}\,(y-1)^{\,k|A|-r(A)}.
\]
\end{definition}

Note that this is in fact a polynomial, as $r(A) \leq k|A|$, for all $A \subseteq E$ and $r$ takes nonnegative integer values.
Moreover, it is a specialization of the slightly more general function that has been defined in Section 7 of \cite{chavez2021new}.
The authors of \cite{chavez2021new} define for a polymatroid and weight function $\omega: E \to \mathbb{R}$, the function
\begin{equation}\label{eq:chavez}
    N(P,\omega;u,v) =\sum_{A \subseteq E} u^{r(E) -r(A)}v^{\sum_{e \in A}\omega(e) -r(A)}.
\end{equation}

Our polynomial $T_k$ is obtained from $N(P,\omega;u,v)$ (after a change of variables)  by taking constant weights $\omega(e)=k$ for all $e\in E$. 
Since we work throughout within the class of $k$-polymatroids, this constant choice is natural and carries no extra data, i.e. the parameter $k$ is already part of the structure, so no separate weight function needs to be specified.
This convention of taking constant weights $\omega(e) = k$ (so that no additional weight data is required) is also consistent with the hypergraph framework that we develop later. 
There, we consider a version of the Tutte polynomial in which the exponent of $(y-1)$ is given by a degree term minus the rank.
For $(k+1)$-uniform hypergraphs this degree term becomes $k|A|$, so in that case the hypergraph Tutte polynomial specializes to $T_k(P_H;x,y)$, where $P_H$ is the associated polymatroid of the $(k+1)-$uniform hypergraph.
Moreover, in our framework, the class of polymatroids associated to $(k+1)-$uniform hypergraphs is closed under deletion and contraction, so this specialization is compatible with the minor structure on both hypergraphs and $k$-polymatroids.

The following proposition, which gives a deletion-contraction reduction, follows directly from Proposition 7.1 of \cite{chavez2021new}.
\begin{proposition}\label{prop: dc-Tpoly}
     Let $P=(E,r)$ be a $k$-polymatroid, and let $e \in E$. Then
    \[
    T_{k}(P;x,y) = (x-1)^{r(E)- r(E\setminus{e})} T_{k}(P\setminus{e};x,y) + (y-1)^{k -r(e)}T_{k}(P/e;x,y).
    \]
\end{proposition}

As in the classical graph and matroid settings, $T_k$ also satisfies a natural multiplicativity property with respect to direct sums.  
\begin{proposition}\label{prop: multip-tutte-kpoly}
    The $k$-polymatroid Tutte polynomial is multiplicative on direct sums.
\end{proposition}

\begin{proof}
    Let $P_1=(E_1,r_1)$ and $P_2=(E_2,r_2)$  be $k$-polymatroids on disjoint groundsets $E_1,E_2$, and let
$P = P_1 \oplus P_2 = (E,r)$ with $E = E_1 \cup E_2$ and
\[
r(A) = r_1(A \cap E_1) + r_2(A \cap E_2) \quad \text{for } A \subseteq E.
\]
Write $A_i \coloneqq A \cap E_i$ for $i=1,2$. 
Using the bijection $A \leftrightarrow (A_1,A_2)$ between subsets of $E$ and pairs
$(A_1 , A_2 )$, we get
\begin{align*}
T_k(P;x,y)
&= \sum_{A_1 \subseteq E_1} \sum_{A_2 \subseteq E_2}
(x-1)^{[r_1(E_1)-r_1(A_1)] + [r_2(E_2)-r_2(A_2)]} \\
&\qquad\cdot (y-1)^{[k|A_1|-r_1(A_1)] + [k|A_2|-r_2(A_2)]} \\
&= \Bigg(\sum_{A_1 \subseteq E_1}
(x-1)^{r_1(E_1)-r_1(A_1)} (y-1)^{k|A_1|-r_1(A_1)}\Bigg) \\
&\qquad \times \Bigg(\sum_{A_2 \subseteq E_2}
(x-1)^{r_2(E_2)-r_2(A_2)} (y-1)^{k|A_2|-r_2(A_2)}\Bigg) \\
&= T_k(P_1;x,y)\,T_k(P_2;x,y).
\end{align*}
\end{proof}

Another key property of $T_k$ is that it is compatible with duality, in analogy with the classical Tutte polynomial, in that taking the $k$-dual simply interchanges the variables.

\begin{proposition}\label{prop:duality k-poly-tutte}
    Let $P=(E,r)$ be a $k$-polymatroid and let $P^*$ be its dual. Then

    \[
    T_k(P^* ;x,y) = T_k(P ;y,x).
    \]
\end{proposition}

\begin{proof}
    \begin{align*}
          T_k(P^* ;x,y) &= \sum_{A \subseteq E}(x-1)^{r^*(E)-r^*(A) }(y-1)^{k|A|-r^* (A)}\\
          &=\sum_{A \subseteq E}(x-1)^{k|E| -r(E) -k|A| -r(E\setminus{A}) +r(E) }(y-1)^{k|A|-k|A| -r(E\setminus{A} )+r(E)}\\
          &=
          \sum_{A \subseteq E} (x-1)^{ k|E\setminus{A}|  -r(E\setminus{A})}(y-1)^{r(E) -r(E\setminus{A})}\\
             &=
          \sum_{A \subseteq E} (x-1)^{ k|A|  -r(A)}(y-1)^{r(E) -r(A)}\\  
          &=T_k(P;y,x).
          \end{align*}
\end{proof}

In \cite{chavez2021new}, the authors also formulate a duality statement for the polynomial $N$ in \ref{eq:chavez}, but in a different setting, namely for the class of \emph{compact polymatroids}, and based on a slightly different concept of polymatroid duality.

We conclude this subsection by relating the $k$-polymatroid Tutte polynomial to the \emph{characteristic polynomial} of a polymatroid (see \cite{whittle1994critical}), defined by
\begin{align*}
    \chi(P;x) = \sum_{A \subseteq E(P)} (-1)^{|A|}x^{r(P)-r(A)}.
\end{align*}

\begin{proposition}\label{prop: char-tutte}
    Let $P$ be a $k$-polymatroid. 
    \begin{enumerate}[i.]
        \item If $k$ is odd, then
    \begin{align*}
        (-1)^{r(P)}T_k(P;x+1,0) = \chi(P;-x).
    \end{align*}
    \item If $k$ is even, then for any $\eta \in \mathbb{C}$ such that  $\eta^k =-1$, we have
    \begin{align*}
       (-\eta)^{r(P)} T_k(P;1+x\eta^{-1},1-\eta)=\chi(P;-x).
    \end{align*}
    \end{enumerate}
    
\end{proposition}
\begin{proof}

For the first part, when $k$ is odd, we simply verify

\begin{align*}    (-1)^{r(P)}T_k(P;x+1,0) &= (-1)^{r(P)}\sum_{A\subseteq E}
x^{\,r(P)-r(A)}\,(-1)^{\,k|A|-r(A)}\\
&=\sum_{A\subseteq E}
(-x)^{\,r(P)-r(A)}\,(-1)^{|A|}\\
&=\chi(P;-x),
\end{align*}
as $(-1)^k=-1$.

If $k$ is even, we will use the fact that $(-\eta)^{k|A|} = (-1)^{|A|}$, and write out
\begin{align*}
    (-\eta)^{r(P)} T_k(P;1+x\eta^{-1},1-\eta)&= (-\eta)^{r(P)}\sum_{A\subseteq E}
(x\eta^{-1})^{\,r(P)-r(A)}\,(-\eta)^{\,k|A|-r(A)}\\
&=
\sum_{A\subseteq E}
x^{\,r(P)-r(A)}\,(-\eta)^{\,k|A|}(-\eta)^{-r(A)}(\eta^{-1})^{\,r(P)-r(A)} (-\eta)^{r(P)}\\
&=\sum_{A\subseteq E}
(-x)^{\,r(P)-r(A)}\,(-1)^{|A|}\\
&=\chi(P;-x),
\end{align*}
as claimed.
    
\end{proof}

\subsection{Universality}
In this subsection, we generalize the universality theorem of the Tutte polymatroid for matroids to $k$-polymatroids. 
Let $P=(E,r)$ be a $k$-polymatroid.
For $e\in E$, denote $\Delta(e) \coloneqq r(P)- r(P\setminus{e})$ and $m(e) \coloneqq k - r(e)$.
We say that $e$ is of \emph{type} $(s,t)$ if $\Delta(e)=s$ and $m(e)=t$.

\begin{theorem}[Universality]\label{thm: poly-universality}
Let $\mathcal{P}_k$ be the class of $k$–polymatroids, and let $\mathcal{P}\subseteq \mathcal{P}_k$ be a minor closed subclass.
Then there exists a unique map 
\[
U:\mathcal{P} \longrightarrow \mathbb{Z}[x,y,a,b]
\]
such that for every $P\in\mathcal{P}$ and every $e\in E(P)$,
\begin{equation}\label{eq: reduc-form-poly}
U(P)=
\begin{cases}1, 
    & \text{if }P=\emptyset,\\[6pt]
a^{\,k-\Delta(e)}x^{\,\Delta(e)}\,U(P\setminus e)
  + b^{\,k-m(e)}y^{\,m(e)}\,U(P/e),
    & \text{otherwise.}
\end{cases}
\end{equation}

Moreover,
\begin{equation}\label{eq:univpoly}
    U(P;x,y,a,b) =a^{k|E|-r(P)}b^{r(P)}T_k\left(P; \frac{ x+b}{b},\frac{y+a}{a}\right),
\end{equation}

where $T_k$ denotes the $k$-polymatroid Tutte polynomial, introduced in Definition~\ref{def: k-polyTutte}.
\end{theorem}

\begin{proof}
Let
\[
\hat U(P):=\beta(P)\,T_{k}\!\left(P;\frac{x+b}{b},\frac{y+a}{a}\right),
\ \text{ where } \
\beta(P):=a^{k|E(P)|-r(P)}b^{r(P)}.
\]
Since $\mathcal P$ is minor-closed, for any $e\in E(P)$ we have $P\setminus e\in\mathcal P$ and $P/e\in\mathcal P$.

\smallskip
\emph{Uniqueness.}
If $U$ satisfies the stated reduction rules of Equation~\ref{eq: reduc-form-poly} , then for any nonempty $P$ and any $e\in E(P)$ the value $U(P)$ is determined by $U(P\setminus e)$ and $U(P/e)$, and iterating reduces to the empty polymatroid. Hence such a map is unique.

\smallskip
\emph{Existence.}
As $\hat U$ exists, it suffices to show that $\hat U$ satisfies the same reduction formula as Equation~\ref{eq: reduc-form-poly}.

If $P=\emptyset$, then $|E(P)|=0$ and $r(P)=0$, so
\[
\hat U(\emptyset)=a^{0}b^{0}\,T_{k}\!\left(\emptyset;\frac{x+b}{b},\frac{y+a}{a}\right)=1.
\]

Now let $P$ be nonempty and fix $e\in E(P)$. Using the deletion--contraction relation for $T_k$ as in Proposition~\ref{prop: dc-Tpoly} ,
\[
T_{k}(P;X,Y)=(X-1)^{\Delta(e)}T_{k}(P\setminus e;X,Y)+(Y-1)^{m(e)}T_{k}(P/e;X,Y),
\]
and substituting $X=\frac{x+b}{b}$ and $Y=\frac{y+a}{a}$, we get
\begin{align*}
\hat U(P)
&=\beta(P)\left(\frac{x}{b}\right)^{\Delta(e)}T_{k}\!\left(P\setminus e;\frac{x+b}{b},\frac{y+a}{a}\right)
 +\beta(P)\left(\frac{y}{a}\right)^{m(e)}T_{k}\!\left(P/e;\frac{x+b}{b},\frac{y+a}{a}\right) \\
&=\frac{\beta(P)}{\beta(P\setminus e)}\left(\frac{x}{b}\right)^{\Delta(e)}\hat U(P\setminus e)
 +\frac{\beta(P)}{\beta(P/e)}\left(\frac{y}{a}\right)^{m(e)}\hat U(P/e).
\end{align*}
We verify for edges of type $(\Delta(e), m(e))$  that
\[
\frac{\beta(P)}{\beta(P\setminus e)}=a^{k-\Delta(e)}b^{\Delta(e)}
\quad\text{and}\quad
\frac{\beta(P)}{\beta(P/e)}=a^{m(e)}b^{k-m(e)}.
\]
Therefore,
\[
\hat U(P)=a^{k-\Delta(e)}x^{\Delta(e)}\hat U(P\setminus e)+b^{k-m(e)}y^{m(e)}\hat U(P/e),
\]
which is exactly the required reduction formula.

Thus $\hat U$ satisfies the stated recursion, and by uniqueness we must have $U=\hat U$.
\end{proof}

\begin{corollary}\label{cor: prod-universal}
  The map $U$ from Theorem \ref{thm: poly-universality} is multiplicative on the direct sum.
\end{corollary}
\begin{proof}
    We remark that 
    \[
    a^{k|E_1 \cup E_2|-r(P_1 \oplus P_2)}b^{r(P_1 \oplus P_2)} = a^{k|E_1| -r(P_1)}b^{r(P_1)}\cdot  a^{k|E_2| -r(P_2)}b^{r(P_2)}.
    \]
    The result now follows from Proposition \ref{prop: multip-tutte-kpoly}
\end{proof}

\begin{remark}
Theorem~\ref{thm: poly-universality} may also be read as the corresponding recipe theorem.
Indeed, let $f$ be an invariant on a minor closed subclass $\mathcal{P} \subseteq \mathcal{P}_k$ satisfying the reduction rule \eqref{eq: reduc-form-poly} for some choice of $x,y,a,b$, together with the prescribed value on the empty polymatroid.
Then Theorem~\ref{thm: poly-universality} implies that $f$ is uniquely determined, and is given by the specialization of $T_k$ in \eqref{eq:univpoly}.
\end{remark}

We next isolate the two extremal types.
We denote an element of a $k$-polymatroid of type $(0,k)$  as a  \emph{$k$-loop}, and an element of type $(k,0)$ is known as a \emph{$k$-coloop}.
\begin{lemma}
If $e$ is a $k$-loop or a $k$-coloop in a $k$–polymatroid $P=(E,r)$, then the minors $P\setminus e$ and $P/ e$ are equal as polymatroids.
\end{lemma}

\begin{proof}
    The polymatroids $P/ e$ and $P\setminus e$ have the same groundsets, namely $E\setminus{e}$.
    Let $A \subseteq E\setminus{e}$.
    First we remark that
    \begin{align*}
              r_{P/e}(A) = r(A \cup e) - r(e) \leq r(A) + r(e) - r(e) = r(A) =r_{P \setminus{e}}(A)
    \end{align*}
    where the penultimate step follows from submodularity.
    
    Now, let $e$ be a $k$-loop. 
    It holds that $r(e)=0$. We obtain for $A\subseteq E\setminus{e}$ 
    \begin{align*}
        r_{P/e}(A) = r(A \cup e) - r(e)
        = r(A \cup e) 
        \geq r(A) = r_{P\setminus{e}}(A).
    \end{align*}
      Thus  $P/ e$ and $P\setminus e$ are equal as polymatroids.

If $e$ is a $k$-coloop, then $r(P) -r(P\setminus{e})=k$.

From submodularity it follows that $r(A \cup {e}) + r(E \setminus{e}) \geq r(E)  +r(A)$. 
Now we obtain
\begin{align*}
    r(A \cup e) \geq r(E) -  r(E \setminus{e}) + r(A) = k + r(A) \geq r(e) +r(A).
\end{align*}

This gives us
\begin{align*}
 r_{P/e}(A) = r(A \cup e) - r(e) \geq r(A) = r_{P\setminus{e}}(A),
\end{align*}
finishing the proof.
  
\end{proof}

\begin{corollary}
    
For the map $U$ from Theorem \ref{thm: poly-universality}, it holds that
\[
U(P)=
\begin{cases}
\left(x^{k}+ b^{k}\right)\,U(P/ e),
    & \text{if $e$ is a $k$-coloop,}\\[6pt]
\left(y^{k}+a^{k}\right)\,U(P\setminus{e}), 
   & \text{if $e$ is a $k$-loop.}\\[6pt]
\end{cases}
\]

\end{corollary}

\subsection{Convolution}
In this subsection, we give a $k$-polymatroid analogue of the convolution product formula for the matroid Tutte polynomial from \cite{kook1997convolution}.

Let $\mathcal{P}_k$ denote the class of $k$-polymatroids, and let $R$ be a commutative ring with unity.
For any functions $f,g: \mathcal{P}_k \to R$, we define the convolution product for a $k$-polymatroid $P=(E,r)$ by
\[
(f\circ g )(P)\;=\;\sum_{A\subseteq E} f(P|_A)\,g(P/A).
\]

We define $\delta:\mathcal{P}_k \to R$ by
\[
\delta(P)=
\begin{cases}
1 & \text{if } P=\emptyset,\\
0 & \text{otherwise}.
\end{cases}
\]

\begin{proposition}
    The convolution product on $k$-polymatroids is an associative operation, with identity $\delta$.
\end{proposition}

\begin{proof}
We quickly verify that
\begin{align*}
(f\circ \delta )(P)
&=\sum_{A\subseteq E} f(P|_A)\,\delta(P/A)
=\;f(P),\\
(\delta \circ f)(P)
&=\sum_{A\subseteq E} \delta(P|_A)\,f(P/A)
=\;f(P).
\end{align*}

The proof of associativity of $\circ$ is the same as in \cite{Knapp2018}.
\begin{align*}
[(f\circ g)\circ h](P)
  &= \sum_{A\subseteq E(P)} (f\circ g)(P|_A)\; h(P/A)\\
  &= \sum_{A\subseteq E(P)} \sum_{B\subseteq E(P|_A)} 
        f\bigl((P|_A)|_B\bigr)\; g\bigl((P|_A)/B\bigr)\; h(P/A)\\
  &= \sum_{B\subseteq E(P)} \sum_{C\subseteq E(P/B)} 
        f\bigl((P|_B)\bigr)\;
        g\bigl((P|_{B\cup C})/B\bigr)\;
        h\bigl(P/(B\cup C)\bigr)
        \quad\text{(where }A=B\cup C\text{)}\\
  &= \sum_{B\subseteq E(P)} \sum_{C\subseteq E(P/B)} 
        f(P|_B)\;
        g\bigl((P/B)|_C\bigr)\;
        h\bigl((P/B)/C\bigr)\\
  &= \sum_{B\subseteq E(P)} 
        f(P|_B)\; (g\circ h)(P/B)\\
  &= [\,f\circ (g\circ h)\,](P).
\end{align*}
\end{proof}

We state our $k$-polymatroid version of the convolution product formula for the Tutte polynomial.

\begin{theorem}\label{thm: conv}
Let $P=(E,r)$ be a $k$-polymatroid. 
    \begin{enumerate}[i.]
        \item If $k$ is odd, then the $k$-polymatroid Tutte polynomial satisfies 
        \begin{align*}
            T_k(P;x+1,y+1)= \displaystyle\sum\limits_{A \subseteq E} T_k(P|_A;0,y+1) T_k(P/A;x+1,0) .
        \end{align*}
        \item If $k$ is even, then the $k$-polymatroid Tutte polynomial satisfies both
        \begin{align*}        T_k(P;x+1,y+1)&=\displaystyle\sum\limits_{A \subseteq E}(-1)^{r(P|_A)}\chi\left((P|_A)^*;-y \right)T_k(P/A;x+1,0),
        \end{align*}
        and,  
        \begin{align*}            T_k(P;x+1,y+1)&=\displaystyle\sum\limits_{A \subseteq E}(-1)^{r(P/A)}T_k\left(P|_A;0,y+1\right)  \chi(P/A;-x),
        \end{align*}
        where $\chi(P;x)=\sum\limits_{A \subseteq E(P)}(-1)^{ |A|} x^{ r(P)-r(A)}$ is the polymatroid characteristic polynomial. 
    \end{enumerate}
\end{theorem}

In order to prove the theorem, we will first introduce the function $\zeta:\mathcal{P} \to \mathbb{C}[x,y] $, defined by $\zeta(x,y)(P) = x^{r(P)}y^{r^*(P)}$, and prove some of its properties.
Note that $T_k$ can be expressed in terms of $\zeta$, as we have

\begin{align*}
    \zeta(1,y) \circ \zeta(x,1) (P) &= \sum_{A \subseteq E(P)} x^{r_{P/A}(P/A)}y^{ r_{P|_A}^*(P|_A)}\\
    &=\sum_{A \subseteq E(P)}x^{r_P(P)-r_P(A)}y^{k|A| +r_{P|_A}(A \setminus{A} ) -r_{P|_A}(A) }\\
    &=\sum_{A \subseteq E(P)}x^{r_P(P)-r_P(A)}y^{k|A| -r_P(A)}\\
    &=T_k(P;x+1,y+1).
\end{align*}

In particular, $T_k(P;x+1,0)=\zeta(1,-1) \circ \zeta(x,1) (P) $ and $T_k(0,y+1)=\zeta(1,y)\circ \zeta(-1,1)(P)$.

\begin{lemma}\label{lem: zetainverse}
    Let $P$ be a $k$-polymatroid, and let $\zeta(x,y)(P) = x^{r(P)}y^{r^*(P)}$. 
    Then for all $\eta \in \mathbb{C}$ such that $\eta^k=-1$, it holds that
    \[
    \zeta(x,y) \circ \zeta(\eta x,\eta y)(P) =  \zeta(\eta x,\eta y) \circ \zeta(x,y) (P)=\delta(P)
    \]

\end{lemma}

\begin{proof}
    We see that
    \begin{align*}
        \zeta(x,y) \circ \zeta(\eta x, \eta y)(P) &= \sum_{A \subseteq E(P)}x^{r(P|_A)}y^{r^*(P|_A)}(\eta x)^{r(P/A)}(\eta y)^{r^*(P/A)}\\
        &=\sum_{A \subseteq E(P)} x^{ r(P|_A) + r(P/A)} y^{r^*(P|_A)+r^*(P/A)} \eta^{r(P/A)+r^*(P/A)}\\
        &= x^{r(P)}y^{ r^* (P)} \sum_{A \subseteq E(P)}\eta ^{ k(|P|-|A|)}\\
        &=x^{r(P)}y^{ r^* (P)} \sum_{A \subseteq E(P)}\eta ^{k|A|}\\
        &=x^{r(P)}y^{ r^* (P)} \sum_{A \subseteq E(P)}(-1)^{|A|}\\
        &=x^{r(P)}y^{ r^* (P)} (1-1)^{|P|}\\
        &= \delta(P).
    \end{align*}
    We used the bijection $A \leftrightarrow P \setminus{A}$ and the binomial theorem.

 Proving $\zeta(\eta x,\eta y)\circ \zeta(x,y)(P)=\delta(P)$ proceeds analogously.

\end{proof}

\begin{lemma}\label{lem: zetaproduct}

Let $P \in \mathcal{P}_k$ and let $\eta \in \mathbb{C}$ be such that $\eta^{k}=-1$.
Then
\begin{align*}
  \zeta(-\eta,\eta)\circ \zeta(\eta,-\eta) (P)
  =
  \begin{cases}
    2^{|P|}(-1)^{|P|+r(P)}, & \text{if $k$ is even},\\[4pt]
    \delta(P),              & \text{if $k$ is odd}.
  \end{cases}
\end{align*}
\end{lemma}

\begin{proof}
If $k$ is even, we obtain
\begin{align*}
     \zeta(-\eta,\eta)\circ \zeta(\eta,-\eta) (P) &= \sum_{A \subseteq E(P)}(-\eta)^{r(P|_A)}\eta^{r^*(P|_A)} \eta^{r(P/A)}(-\eta)^{r^*(P/A)}\\
     &= \sum_{A \subseteq E(P)}(-\eta)^{r(A)}\eta^{k|A|-r(A)} \eta^{r(P)-r(A)}(-\eta)^{k(|P|-|A|) -r(P)+r(A)}\\
     &=\sum_{A \subseteq E(P)} \eta^{k|P|} (-1)^{k(|P|-|A|)-r(P) }\\
     &= \sum_{A \subseteq E(P)}(-1)^{|P|} (-1)^{-r(P) }\\
     &=2^{|P|}(-1)^{|P|+r(P)},
\end{align*}

where we have used that $\eta^{k}=-1$ and $(-1)^{k}=1$.

If $k$ is odd, then we get
\begin{align*}
     \zeta(-\eta,\eta)\circ \zeta(\eta,-\eta) (P)
     &=\sum_{A \subseteq E(P)} \eta^{k|P|} (-1)^{k(|P|-|A|)-r(P) }\\
     &= \sum_{A \subseteq E(P)} (-1)^{-|A|-r(P) }\\
     &=\delta(P).
\end{align*}

\end{proof}

\begin{proof}[Proof of Theorem \ref{thm: conv}]

The proof for the case when $k$ is odd is almost identical as in Theorem 1 of \cite{kook1997convolution}. Let $\eta \in \mathbb{C}$, such that $\eta^{k}=-1$.
Applying Lemma \ref{lem: zetainverse}, Lemma \ref{lem: zetaproduct}, and associativity of the convolution product, gives us

\begin{align*}
    T_k(P;x+1, y+1)&=
    \zeta(1,y) \circ \zeta(x,1) (P)\\
    &=
    \zeta(1,y) \circ \left[  \zeta(-1,1) \circ \zeta(-\eta,\eta) \right] \circ \left[ \zeta(\eta,-\eta) \circ  \zeta(1,-1) \right]  \circ \zeta(x,1) (P)\\
    &= T_k(0,y+1) \circ \left[  \zeta(-\eta,\eta)\circ \zeta(\eta,-\eta)  \right] \circ T_k(x+1,0) (P)\\
    &=T_k(0,y+1) \circ \delta \circ T_k(x+1,0) (P)\\
    &=T_k(0,y+1) \circ T_k(x+1,0) (P).
\end{align*}

    Now, let $k$ be even. We have

\begin{align*}
    T_k(P;x+1, y+1)&=
    \zeta(1,y) \circ \zeta(x,1) (P)\\
    &= T_k(0,y+1) \circ \left[  \zeta(-\eta,\eta)\circ \zeta(\eta,-\eta)  \right] \circ T_k(x+1,0) (P)\\
    &=T_k(0,y+1) \circ g \circ T_k(x+1,0) (P),
\end{align*}
where $g(P) = 2^{|P|}(-1)^{|P|+r(P)}$. 

We now compute $T_k(0,y+1)\circ g(P)$.  
Starting from the convolution definition, we have
\[
T_k(0,y+1)\circ g(P)
=
\sum_{A\subseteq E(P)} T_k(P|_A;0,y+1)\,2^{|P/A|}(-1)^{|P/A|+r(P/A)}.
\]
Using the subset expansion of $T_k$ at $(0,y+1)$
\[
T_k(P|_A;0,y+1)
= \sum_{B\subseteq A}(-1)^{r(A)-r(B)}\,y^{k|B|-r(B)},
\] 
and with the identities $|P/A|=|P|-|A|$ and $
r(P/A)=r_{P/A}(E(P)\setminus A)=r(P)-r(A)$,
where the last two ranks are taken in $P$, we get
\begin{align*}
T_k(0,y+1)\circ g(P)
&= \sum_{A\subseteq E(P)}\sum_{B\subseteq A}
(-1)^{r(A)-r(B)}y^{k|B|-r(B)}\,
2^{|P|-|A|}(-1)^{|P|-|A|+r(P)-r(A)}\\
&= \sum_{A\subseteq E(P)}\sum_{B\subseteq A}
2^{|P|-|A|}(-1)^{|P|-|A|+r(P)-r(B)}y^{k|B|-r(B)}\\
&= (-2)^{|P|}\sum_{A\subseteq E(P)}\sum_{B\subseteq A}
(-2)^{-|A|}(-1)^{r(P)-r(B)}y^{k|B|-r(B)}.
\end{align*}
Swapping the order of summation yields
\begin{align*}
T_k(0,y+1)\circ g(P)
&= (-2)^{|P|}\sum_{B\subseteq E(P)}
(-1)^{r(P)-r(B)}y^{k|B|-r(B)}
\sum_{\substack{A\\ B\subseteq A\subseteq E(P)}}(-2)^{-|A|}.
\end{align*}
For the inner sum, write $A=B\cup C$ with $C\subseteq E(P)\setminus B$, giving
\begin{align*}
\sum_{\substack{A\\ B\subseteq A\subseteq E(P)}}(-2)^{-|A|}
&=\sum_{C\subseteq E(P)\setminus B}(-2)^{-|B|-|C|}\\
&= (-2)^{-|B|}\sum_{C\subseteq P\setminus B}(-2)^{-|C|}\\
&=(-2)^{-|B|}\sum_{j=0}^{|P|-|B|}
\binom{|P|-|B|}{j}\left(-\tfrac12\right)^{j}\\
&=(-2)^{-|B|}\left(1-\tfrac12\right)^{|P|-|B|}\\
&=2^{-|P|}(-1)^{|B|}.
\end{align*}
Substituting this gives
\begin{align*}
T_k(0,y+1)\circ g(P)
&= (-2)^{|P|}\sum_{B\subseteq E(P)}
(-1)^{r(P)-r(B)}y^{k|B|-r(B)}\cdot 2^{-|P|}(-1)^{|B|}\\
&= (-1)^{|P|}\sum_{B\subseteq E(P)}
(-1)^{r(P)-r(B)}y^{k|B|-r(B)}(-1)^{|B|}.
\end{align*}
By introducing $\eta \in \mathbb{C}$, with $\eta^k=-1$, we note that
\begin{align*}
T_k(0,y+1)\circ g(P)
&= (-1)^{|P|}\sum_{B\subseteq E(P)}
(-1)^{r(P)-r(B)}y^{k|B|-r(B)}(-1)^{|B|}\\
&= (-1)^{|P|}\sum_{B\subseteq E(P)}
(-1)^{r(P)-r(B)}( y\eta^{-1})^{k|B|-r(B)}\eta^{k|B|-r(B)}(-1)^{|B|}\\
\end{align*}
Since $\eta^{k}=-1$ we have $\eta^{k|B|}(-1)^{|B|}=1$

Thus we have
\begin{align*}
T_k(0,y+1)\circ g(P)
&= (-1)^{|P|}\sum_{B\subseteq E(P)}
\eta^{-r(B)}(-1)^{r(P)-r(B)}(y\eta ^{-1})^{k|B|-r(B)}\\
&=(-1)^{|P|}\eta^{ -r(P)}\sum_{B\subseteq E(P)}
\eta^{r(P)-r(B)}(-1)^{r(P)-r(B)}(y\eta ^{-1})^{k|B|-r(B)}
\end{align*}
We recognize the expression of the
remaining sum as  $(-1)^{ |P|} \eta^{-r(P)} T_k(P;1- \eta, 1+  y\eta^{-1})$.

Using the duality property from Proposition~\ref{prop:duality k-poly-tutte} and Proposition~\ref{prop: char-tutte}, we get
\begin{align*}
    (-1)^{ |P|} \eta^{-r(P)} T_k(P;1- \eta, 1+  y\eta^{-1})& =(-1)^{ |P|} \eta^{-r(P)} T_k(P^*; 1+  y\eta^{-1}, 1- \eta)\\
    &=(-1)^{ |P|} \eta^{-r(P)}  (-\eta)^{ -r(P*)}\chi(P^*;-y)
\end{align*}

We see that $(-\eta)^{-r(P^*)}=(-\eta)^{-k|P| +r(P)} = (-1)^{|P|+r(P)}\eta^{r(P)}$, giving us that
\begin{align*}
    T_k(0,y+1)\circ g(P) &= (-1)^{r(P)}\chi(P^*;-y).
\end{align*}
Thus
\begin{align*}
     T_k(P;x+1, y+1)&=
  T_k(0,y+1) \circ g \circ T_k(x+1,0) (P)\\
    &= (-1)^{r(P)}\chi^*(-y)\circ T_k(x+1,0) (P) \\
&=\displaystyle\sum\limits_{A \subseteq E(P)}(-1)^{r(P|_A)}\chi\left((P|_A)^*;-y \right)T_k(P/A;x+1,0),
\end{align*}
where $\chi^*(x):  P \mapsto  \chi(P^*;x)$.

For the second equality of ii. of Theorem \ref{thm: conv}, we make use of the associativity property of the convolution product, and compute $g \circ T_k(x+1,0) (P)$ instead.

We have  
\[
g\circ T_k(x+1,0)(P)
=
\sum_{A\subseteq E(P)} 2^{|A|}(-1)^{|A|+r(A)}\,T_k(P/A;x+1,0).
\]

Using the subset expansion of $T_k$ at $(x+1,0)$,
\[
T_k(P/A;x+1,0)
= \sum_{B\subseteq E(P)\setminus A}
x^{\,r(P/A)-r_{P/A}(B)}(-1)^{k|B|-r_{P/A}(B)},
\]
and the identities
\[
r(P/A)=r(P)-r(A),
\qquad
r_{P/A}(B)=r(A\cup B)-r(A),
\]
we obtain
\begin{align*}
g\circ T_k(x+1,0)(P)
&= \sum_{A\subseteq E(P)}\sum_{B\subseteq E(P)\setminus A}
2^{|A|}(-1)^{|A|+r(A)}
x^{r(P)-r(A)-(r(A\cup B)-r(A))}
(-1)^{k|B|-r(A\cup B)+r(A)}\\[4pt]
&= \sum_{A\subseteq E(P)}\sum_{B\subseteq E(P)\setminus A}
2^{|A|}
x^{r(P)-r(A\cup B)}(-1)^{|A|+ k|B|-r(A\cup B)}.
\end{align*}

The sum runs over all pairs $(A,B)$, with $A,B \subseteq E(P)$, such that $A \cap B = \emptyset$.
By setting $C=A \cup B$, we see that this is equivalent to summing over all pairs $(A,C)$, where $A \subseteq C$ (and $B=C\setminus{A}$).
Thus
\begin{align*}
g\circ T_k(x+1,0)(P)
&= \sum_{C\subseteq E(P)}\sum_{A\subseteq C}x^{r(P)-r(C)} \,
2^{|A|}(-1)^{|A|+k|C\setminus A|-r(C)}
\\
&=\sum_{C\subseteq E(P)}
x^{r(P)-r(C)}(-1)^{k|C|-r(C)}
\sum_{A\subseteq C}2^{|A|}(-1)^{(1-k)|A|}.
\end{align*}

Since $k$ is even, $1-k$ is odd, and therefore $(-1)^{(1-k)|A|}=(-1)^{|A|}$. Thus, using the binomial theorem,  the inner sum becomes
\begin{align*}
\sum_{A\subseteq C}2^{|A|}(-1)^{|A|}=(-1)^{|C|}.
\end{align*}
We therefore have
\begin{align*}
g\circ T_k(x+1,0)(P)
&= \sum_{C\subseteq E(P)}
x^{r(P)-r(C)}\,(-1)^{k|C|-r(C)}(-1)^{|C|}.
\end{align*}

Now introduce $\eta\in\mathbb{C}$ with $\eta^k=-1$ to get
\begin{align*}
    g\circ T_k(x+1,0)(P)
&= \sum_{C\subseteq E(P)}
x^{r(P)-r(C)}\,(-1)^{k|C|-r(C)}(-1)^{|C|}\\
&= \sum_{C\subseteq E(P)}
x^{r(P)-r(C)}\,(-1)^{k|C|-r(C)}\eta^{k|C|}\\
&=\sum_{C\subseteq E(P)}
x^{r(P)-r(C)}\,(-1)^{k|C|-r(C)}\eta^{k|C|-r(C)} \eta^{r(C)}
\end{align*}

Rewrite $\eta^{r(C)}= \left(\frac{1}{\eta}\right)^{ r(C)}$ and factor out a global $\left( \frac{1}{\eta}\right) ^{-r(P)}=\eta^{r(P)}$, giving
\begin{align*}
 g\circ T_k(x+1,0)(P) &=\eta^{r(P)}\sum_{C\subseteq E(P)}
x^{r(P)-r(C)}\,(-1)^{k|C|-r(C)}\eta^{k|C|-r(C)} \left(\frac{1}{\eta}\right)^{r(P)-r(C)}\\
&=\eta^{r(P)}\sum_{C\subseteq E(P)}
\left(x \eta^{-1}\right)^{r(P)-r(C)}\,(-\eta)^{k|C|-r(C)}.\\
\end{align*}

We recognize the remaining sum as the subset expansion of the $k$-polymatroid Tutte polynomial, evaluated at $(1+x\eta^{-1} ,\,1-\eta)$. Therefore, by Proposition \ref{prop: char-tutte}, 
\begin{align*}
    g\circ T_k(x+1,0)(P)
&= \eta^{r(P)}\,T_k\bigl(P;\,1+x\eta^{-1} ,\,1-\eta \bigr)\\
&=(-1)^{r(P)}\chi(P;-x)
\end{align*}
as claimed.
\end{proof}

\section{Hypergraphs}\label{sec: Hypergraphs}
In this section, we return to hypergraphs and the hypergraph Tutte polynomial given in Definition \ref{def: HyperTutte}.

We relate hypergraphs to polymatroids in a common way and show that our deletion–contraction operations on hypergraphs are compatible with those on the associated polymatroids.
In particular, we will realize $(k+1)$-uniform hypergraphs as $k$-polymatroids, and apply the results from Section~\ref{sec: polymatroids} to the setting of uniform hypergraphs.

\subsection{Hypergraphs as Polymatroids}

Let $H=(V,E)$ be a hypergraph. 
Its \emph{associated polymatroid} $P_H$ is the polymatroid on the groundset $E$, with rank function $r(A)= v(H) - \kappa_H(A)$ for $A \subseteq E$.
It is well known that this indeed gives a polymatroid \cite{VertigPoly}.
The multiplicity of edges does not affect the rank function.
Allowing the multiplicities, in combination of the definition of deletion and contraction as in Definition~\ref{def: del/con}, allows us to state the following proposition.

\begin{proposition} \label{prop: hyper-poly-bridge}
    Hypergraphs and their associated polymatroids have compatible deletion-contraction operations, i.e.  
    $P_{H\setminus{e}}= P_H\setminus{e}$ and $P_{H/e}\cong P_H/e$, where $\cong$ denotes isomorphism of polymatroids (a groundset bijection preserving the rank function).    
    Here, hypergraph deletion and contraction are as in Definition~\ref{def: del/con}, while polymatroid deletion and contraction follow the usual definitions, as recalled in Section~\ref{sec: polymatroids}.
\end{proposition}

\begin{proof}
Let $e \in E$ and consider $H \setminus e$.
    For any $A \subseteq E \setminus{e}$ we have
    \[
        r_{P_{H\setminus e}}(A)
        = v(H\setminus e) - \kappa_{H\setminus e}(A)
        = v(H) - \kappa_H(A)
        = r_{P_H}(A)
        = r_{P_H \setminus e}(A),
    \]
    since deleting $e$ does not change the vertex set, nor the connected
    components of the subhypergraph with edge set $A$. Hence
    $P_{H\setminus e} = P_H \setminus e$.

 For contraction, recall that $H/e$ is obtained by identifying all
    vertices of $e$ to a single vertex, while retaining the incidences of the adjacent edges.
    Thus
    \[
        v(H/e) = v(H) - v(e) + 1,
    \]
    where $v(e)$ is the number of distinct vertices contained in $e$.
    Moreover, there is a natural bijection between $E(H/e)$ and $E(H) \setminus{e}$, and for any $A\subseteq E(H/e)$  we have
    \[
        \kappa_{H/e}(A) = \kappa_H(A \cup e).
    \]
    Hence for $A \subseteq E(H/e)$,
    \begin{align*}
        r_{P_{H/e}}(A)
            &= v(H/e) - \kappa_{H/e}(A) \\
            &= \bigl(v(H) - v(e) + 1\bigr) - \kappa_H(A \cup \{e\}).
    \end{align*}
    On the other hand, by the definition of polymatroid contraction,
    \begin{align*}
        r_{P_H / e}(A)
            &= r_{P_H}(A \cup e) - r_{P_H}(e) \\
            &= \bigl(v(H) - \kappa_H(A \cup e)\bigr)
               - \bigl(v(H) - \kappa_H(e)\bigr) \\
            &= \kappa_H(e) - \kappa_H(A \cup e)\\
            &= \bigl(v(H) - v(e) + 1\bigr) - \kappa_H(A \cup e).\\
    \end{align*}
    Thus, for all $A \subseteq E \setminus\{e\}$, the rank functions agree under the above
    natural bijection of ground sets, and hence $P_{H/e}\cong P_H / e$.
    This proves the proposition.

\end{proof}

\begin{corollary}
    Let $\mathcal{H}$ be a minor-closed class of hypergraphs. Then the associated class of polymatroids $\mathcal{P}_\mathcal{H}=\{P_H : H \in \mathcal{H}\}$ is again minor-closed.
\end{corollary}

Proposition \ref{prop: hyper-poly-bridge} is useful, as deletion-contraction properties of polymatroid invariants can be translated to the hypergraph setting.
Since our main object of interest in the polymatroid setting is our $k$-polymatroid Tutte polynomial, we remark that the polymatroids associated to $(k+1)$-uniform hypergraphs are $k$-polymatroids, as $r(e) =v(e) -1 \leq k$ for any edge of degree $k+1$, where we recall that $v(e)$ is the number of \emph{distinct} vertices in $e$. 
Usually, in the literature, the edges of hypergraphs do not allow multiplicity. 
A $(k+1)$-uniform hypergraph, when not allowing multiplicities, is described as a \emph{strict} $k$-polymatroid, meaning $r(e)=k$ for all $e \in E$. 
The problem, however, is that the class of strict $k$-polymatroids is not closed under deletion and contraction. 
By allowing multiplicities, a $(k+1)$-uniform hypergraph is just a $k-$polymatroid, with a more natural correspondence between the minors.
Therefore, in this framework, the relationship between $k$–polymatroids and $(k+1)$–uniform hypergraphs is cleaner and more direct. 
Moreover, if $H$ is $(k+1)$-uniform, then the hypergraph Tutte polynomial as in Definition \ref{def: HyperTutte} is simply the $k$-polymatroid Tutte polynomial as in Definition \ref{def: k-polyTutte} applied to the associated polymatroid of the hypergraph. 
Indeed, recall that the hypergraph Tutte polynomial as in \ref{def: HyperTutte} is given by
\begin{equation} \label{THG2}  T_{\mathrm{HG}}(H;x,y) =  \sum_{A \subseteq E(H)} (x-1)^{\kappa(A)- \kappa(H) } (y-1)^{d(A) - |A| - v(H)+\kappa(A)}.
\end{equation}
For a $(k+1)$-uniform hypergraph $H$, it holds that

\begin{align*}
    T_{\mathrm{HG}}(H;x,y) &=  \sum_{A \subseteq E(H)} (x-1)^{\kappa(A)- \kappa(H) } (y-1)^{d(A) - |A| - v(H)+\kappa(A)}\\
&=\sum_{A\subseteq E}
(x-1)^{\,r(E)-r(A)}\,(y-1)^{\,k|A|-r(A)}\\
&=T_k(P_H;x,y),
\end{align*}

where $T_k$ is as in Definition \ref{def: k-polyTutte}. This provides a direct bridge from uniform hypergraphs to the $k$-polymatroid theory, and we may therefore invoke the properties proved for $T_k$ in Section~\ref{sec: polymatroids}.

\subsection{Universality Theorem and Convolution for Uniform Hypergraphs}

For the remainder of this section, we fix an integer $k \ge 1$ and, unless stated otherwise, work with $(k+1)$-uniform hypergraphs.

In this subsection, we will give a universality theorem for $(k+1)$-uniform hypergraphs, generalizing the well-known universality and recipe theorem for graphs (see \cite{Bol98, ellis2010graph, oxley1979tutte, Wel93}), and we conclude the section by stating a convolution product formula.
We start by stating a slightly easier variation that follows directly from applying the $k$-polymatroid universality theorem from Theorem~\ref{thm: poly-universality} to $(k+1)$-uniform hypergraphs through Proposition~\ref{prop: hyper-poly-bridge}.

\begin{lemma}\label{lem: recipe-univ-hyper-simple}
Let $\mathcal{H}_{k+1}$ be the class of $(k+1)$–uniform hypergraphs, and let $\mathcal{H} \subseteq \mathcal{H}_{k+1}$ be a minor closed subclass.
Then there exists a unique map 
\[
U':\mathcal{H} \longrightarrow \mathbb{Z}[x,y,a,b],
\]
such that for every $H \in \mathcal{H}$ and every $e \in E(H)$,
\[
U'(H)=
\begin{cases}
1, 
    & \text{if }E(H)=\emptyset\text{ and }v(H)=n\\[6pt]
a^{\,k-\Delta(e)}x^{\,\Delta(e)}\,U'(H\setminus e)
  + b^{\,k-m(e)}y^{\,m(e)}\,U'(H/e),
    & \text{otherwise.}
\end{cases}
\]
Here $\Delta(e) = \kappa(H\setminus{e})-\kappa(H)$
and $m(e)=k+1-v(e)$ is the multiplicity,
where $v(e)$ is the number of distinct vertices contained in $e$.

Moreover,
    
\begin{align*}
    U'(H;x,y,a,b) &= a^{k|E|-r(E)}b^{r(E)}\THG\left(H; \frac{ x}{b}+1,\frac{y}{a}+1\right)\\
    &=
    a^{k|E|-v(H)+\kappa(H)}b^{v(H)-\kappa(H)}\THG\left(H; \frac{ x}{b}+1,\frac{y}{a}+1\right),\\
\end{align*}

where $T_{\mathrm{HG}}$ denotes the hypergraph Tutte polynomial introduced in Definition~\ref{def: HyperTutte}.
\end{lemma}

\begin{remark}\label{rem: multip}
    From Corollary~\ref{cor: prod-universal}, it follows that $U'$ is multiplicative over disjoint unions and vertex joins.
\end{remark}

Lemma~\ref{lem: recipe-univ-hyper-simple} covers the case for invariants  that evaluate to $1$ on every edgeless hypergraph.
To generalize the universality theorem for graphs, we have to consider invariants that allow arbitrary values for the singleton hypergraph.
We now state the slightly more general version of the universality theorem for $(k+1)$-uniform hypergraphs.

\begin{theorem}[Universality]\label{thm: univers-unif-hyper}

Let $\mathcal{H}_{k+1}$ be the class of $(k+1)$–uniform hypergraphs, and let $\mathcal{H} \subseteq \mathcal{H}_{k+1}$ be a minor closed subclass.
Then there exists a unique map 
\[
U:\mathcal{H} \longrightarrow \mathbb{Z}[x,y,\alpha,a,b]
\]
such that for every $H \in \mathcal{H}$ and every $e \in E(H) $
\[
U(H)=
\begin{cases}

\alpha^{\,n}, 
    & \text{if }E(H)=\emptyset\text{ and }v(H)=n\\[6pt]
a^{\,k-\Delta(e)}x^{\,\Delta(e)}\,U(H\setminus e)
  + b^{\,k-m(e)}y^{\,m(e)}\,U(H/e),
    & \text{otherwise.}
\end{cases}
\]
Here $\Delta(e) = \kappa(H\setminus{e})-\kappa(H)$
and $m(e)=k+1-v(e)$ is the multiplicity  ,
where $v(e)$ is the number of distinct vertices contained in $e$.

Moreover,
\begin{equation}\label{eq:univ-hyper}
    U(H;x,y,\alpha,a,b) =\alpha^{\kappa(H)}a^{k|E|-v(H)+\kappa(H)}b^{v(H)-\kappa(H)}\THG\left(H; \frac{\alpha x}{b}+1,\frac{y}{a}+1\right),
\end{equation}

where $T_{\mathrm{HG}}$ denotes the hypergraph Tutte polynomial introduced in Definition~\ref{def: HyperTutte}.
\end{theorem}

\begin{proof}

    Similarly to the polymatroid case, uniqueness follows because,  given $e\in E(H)$, $U(H)$ is completely determined by $U(H\setminus{e})$ and $U(H/e)$. 
     It remains to show that the function $U$, given by the right hand side of Equation \ref{eq:univ-hyper} satisfies the reduction formula given in the theorem.
     This will in particular show that $U$ is well-defined.

 Let $U'$ be the map given by Lemma~\ref{lem: recipe-univ-hyper-simple}. 
    We see that
    \begin{align*}
        U(H;x,y,\alpha,a,b) = \alpha^{\kappa(H)}U'(H;\alpha x, y, a, b).
    \end{align*}
   For the empty hypergraph $H$ on $n$ vertices, we quickly verify that indeed, $ U(H;x,y,\alpha,a,b)  = \alpha^n$.
     Now, let the edge set of $H$ be non-empty, and let $e \in E(H)$. 
    Applying Lemma~\ref{lem: recipe-univ-hyper-simple} gives us
    \begin{align*}
         U(H;x,y,\alpha,a,b) =& \alpha^{\kappa(H)}U'(H;\alpha x, y, a, b)\\
         =&\alpha^{\kappa(H)}a^{\,k-\Delta(e)}\alpha^{\Delta(e)}x^{\,\Delta(e)}\,U'(H\setminus e;\alpha x, y, a, b)\\
  &+ \alpha^{\kappa(H)}b^{\,k-m(e)}y^{\,m(e)}\,U'(H/e;\alpha x, y, a, b)\\
  =&\alpha^{\kappa(H)}a^{\,k-\Delta(e)}\alpha^{\Delta(e)}x^{\,\Delta(e)}\,\alpha^{-\kappa(H\setminus e)} U(H\setminus{e};x,y,\alpha,a,b)\\  &+ \alpha^{\kappa(H)} b^{\,k-m(e)}y^{\,m(e)}\, \alpha^{-\kappa(H/e)} U(H/e;x,y,\alpha,a,b).
    \end{align*}

Now, remarking that $\kappa(H)=\kappa(H/e)$ and $\Delta(e)=\kappa(H\setminus{e})-\kappa(H)$, we get
\[
U(H;x,y,\alpha,a,b)=a^{\,k-\Delta(e)}x^{\,\Delta(e)}\,U(H\setminus{e};x,y,\alpha,a,b) +  b^{\,k-m(e)}y^{\,m(e)}\, U(H/e;x,y,\alpha,a,b),
\]
as desired.

\end{proof}

\begin{remark}
As in the polymatroid setting, Theorem~\ref{thm: univers-unif-hyper} may also be read as the corresponding recipe statement.
More precisely, once the parameters $x,y,\alpha,a,b$ are fixed, together with the reduction rule above and the prescribed values $U(H)=\alpha^n$ on edgeless hypergraphs with $n$ vertices, the invariant is uniquely determined.
Equation~\eqref{eq:univ-hyper} then gives the resulting specialization of $\THG$.
\end{remark}

We will now state a few corollaries of  Theorem~\ref{thm: univers-unif-hyper}.
\begin{corollary}
Let $U$ be the map of Theorem~\ref{thm: univers-unif-hyper}. 
Then, for $\alpha=a=b=1$,
\[
U(H;x,y,1,1,1)=T_{\mathrm{HG}}(H;x+1,y+1).
\]
Equivalently, $T_{\mathrm{HG}}(H;x+1,y+1)$ is the unique map satisfying the reduction formulas of Theorem~\ref{thm: univers-unif-hyper} with $(\alpha,a,b)=(1,1,1)$.
\end{corollary}

\begin{corollary}
Let $U$ be the map of Theorem~\ref{thm: univers-unif-hyper}. 
Then $U$ is multiplicative with respect to disjoint unions, that is,
\[
U(H_1 \cup H_2;x,y,\alpha,a,b)
= U(H_1;x,y,\alpha,a,b)\,U(H_2;x,y,\alpha,a,b)
\]
for all disjoint hypergraphs $H_1,H_2$. 

If $\alpha=1$, then $U$ is also multiplicative with respect to one-vertex joins, that is,
\[
U(H_1 * H_2;x,y,1,a,b)
= U(H_1;x,y,1,a,b)\,U(H_2;x,y,1,a,b).
\]
\end{corollary}

\begin{proof}
Let $U'$ be as in Lemma~\ref{lem: recipe-univ-hyper-simple}, which is multiplicative with respect to disjoint unions and vertex joins, as mentioned in Remark~\ref{rem: multip}. 
For disjoint $(k+1)$-uniform hypergraphs $H_1$ and $H_2$ we have
\begin{align*}
    U(H_1 \cup H_2;x,y, \alpha, a,b) &= \alpha^{ \kappa(H_1 \cup H_2)} U'(H_1 \cup H_2; \alpha x, y,a,b)\\
    & =\alpha^{\kappa(H_1)}U'(H_1; \alpha x, y,a,b)\alpha^{\kappa(H_2)}U'(H_2; \alpha x, y,a,b)\\
    &=U(H_1 ;x,y, \alpha, a,b)U(H_2;x,y, \alpha, a,b).
\end{align*}

If $\alpha =1$, then $U$ and $U'$ agree on all $(k+1)$-uniform hypergraphs, and hence $U$ is multiplicative over vertex-joins. 
\end{proof}
 A hyperedge of degree $k+1$ is called a $k$-bridge if $\Delta(e)=k$, and a $k$-loop if $m(e)=k$.
\begin{corollary}\label{cor: bridges-loops}
For the map $U$ from Theorem~\ref{thm: univers-unif-hyper}, it holds that
    \[
U(H)=
\begin{cases}
\left(x^{k}+\left(\frac{b}{\alpha}\right)^{k}\right)\,U(H\setminus e),
    & \text{if $e$ is of type $(k,0)$ (a $k$–bridge),}\\[6pt]
\left(y^{k}+a^{k}\right)\,U(H/e), 
   & \text{if $e$ is of type $(0,k)$ (a $k$–loop).}\\[6pt]
\end{cases}
\]
\end{corollary}

\begin{proof}
    If  $e$ is a $k$-bridge, then $H/e$ and $H \setminus{e}$ have the same associated polymatroid, therefore $U'(H/e)=U '(H\setminus{e})$, where $U'$ is as in Lemma~\ref{lem: recipe-univ-hyper-simple}.
    So, we have that
    \begin{align*}
        U(H\setminus{e};x,y,\alpha,a,b) &= \alpha^{\kappa(H\setminus{e})}U'(H\setminus{e};\alpha x,y,a,b)\\
        & =\alpha^{\kappa(H\setminus{e})}U'(H/e;\alpha x,y,a,b)\\
        &=\alpha^{\kappa(H\setminus{e})}\alpha^{-\kappa(H/e)}U(H/e;x,y, \alpha,a,b).
    \end{align*}
    As $\kappa(H/e)= \kappa(H)$, and $\kappa(H\setminus{e})-\kappa(H)=k$ in this case,
    we get
    \[
     U(H\setminus{e};x,y,\alpha,a,b) = \alpha^{k}U(H/e;x,y,\alpha,a,b).
    \]
    Moreover, for a $k$-bridge, we have $m(e)=0$, so by Theorem~\ref{thm: univers-unif-hyper} we get

    \begin{align*}
        U(H) &= a^{0}x^k U(H\setminus{e}) + b^k y^0 \alpha^{-k}U(H \setminus{e})\\
        &= \left(x^k  + \left(\frac{b}{\alpha}\right)^k \right) U(H \setminus{e}).
    \end{align*}

    If $e$ is a $k$-loop, then $H/e = H\setminus{e}$. 
    Moreover, $\Delta(e)=0$.  
    So in this case we immediately get from Theorem~\ref{thm: univers-unif-hyper} 
    \begin{align*}
        U(H) &= a^{k}x^0 U(H/e) + b^0 y^k U(H /e)\\
        &= \left( y^k +a^k\right) U(H/ e).
    \end{align*}
\end{proof}

We conclude this section with a convolution theorem. 
The convolution product for the $k$-polymatroid Tutte polynomial in Theorem~\ref{thm: conv} gives us directly a convolution product for our hypergraph Tutte polynomial in the uniform setting.  

For a $(k+1)$-uniform hypergraph $H$, define
\[
\chi(H;x)=\sum\limits_{B \subseteq E}(-1)^{|B|} x^{\kappa(B)-\kappa(H)},
\qquad
\chi^*(H;x)=\sum\limits_{B\subseteq E}(-1)^{|B|}\,x^{\,k(|E|-|B|)-v(H)+\kappa(E\setminus B)}.
\]

\begin{theorem}\label{thm: conv hyper}
Let $k \in \mathbb{N}$ and let $H$ be a $(k+1)$-uniform hypergraph. 
    \begin{enumerate}[i.]
        \item If $k$ is odd, then the hypergraph Tutte polynomial satisfies 
        \begin{align*}
            \THG(H;x+1,y+1)= \displaystyle\sum\limits_{A \subseteq E} \THG(H|_A;0,y+1) \THG(H/A;x+1,0) .
        \end{align*}

        \item If $k$ is even, then the hypergraph Tutte polynomial satisfies both 
        \begin{align*}
            \THG(H;x+1,y+1)&=\displaystyle\sum\limits_{A \subseteq E}(-1)^{v(H) - \kappa_H(A) }\chi^*(H|_A;-y)\THG(H/A;x+1,0),
        \end{align*}
        and  
        \begin{align*}
            \THG(H;x+1,y+1)&=\displaystyle\sum\limits_{A \subseteq E}(-1)^{v(H/A) - \kappa(H/A) }\THG(H|_A;0,y+1)  \chi(H/A;-x).
        \end{align*}
    \end{enumerate}
\end{theorem}

\section{Potts Model and Random Cluster Model}\label{sec: potts}
A compelling property of the Tutte polynomial is its relevance in statistical mechanics through its connection with the Potts and random cluster model~\cite{BE-MPS10, WM00}.
The partition function of the $q$-state Potts model on a graph $G=(V,E)$ can be written in the
Fortuin--Kasteleyn form as
\[
Z(G;q,w) = \sum_{A \subseteq E} q^{\kappa(A)}\,(w-1)^{|A|}. 
\]
 
This expression is well-defined, even if $q$ is not an integer.
If $q$ is a positive integer, an equivalent expression is given by the spin-sum representation
\[
 \sum_{\phi: V \to [q]} \prod_{\substack{\{a,b\}\in E(G)\\ \phi(a)=\phi(b)}} w. 
\]
 The Potts model partition function can also be expressed as
\[
Z(G;q,w)  =q^{\kappa(G)}(w-1)^{|V(G)|- \kappa(G) }T\left(G; \frac{q+w-1}{w-1}, w \right),
\]
where $T$ is the Tutte polynomial for graphs.
The partition function for the random cluster model is given by the transformation $w-1 \mapsto v$.

The main goal of this section is to highlight the relevance of the newly introduced hypergraph Tutte polynomial to known and newly introduced hypergraph extensions of the Potts and random cluster model.
We will see that different hypergraph extensions of the same model, distinguished primarily by how they incorporate hyperedge degrees, lead to non-equivalent models in the sense that the corresponding partition functions do not specialize to one another.
Moreover, in the `degree dependent' extensions that we will introduce in this section, the equivalence between the random cluster model and the Potts model ceases to exist, and only the partition function of the former model will be equivalent to the hypergraph Tutte polynomial.
However, when we restrict ourselves to the subclass of uniform hypergraphs, the hyperedge degree will be a constant, and all of the partition functions that we introduce will be equivalent. 
In this uniform case, we may apply the universality theorem (Theorem~\ref{thm: univers-unif-hyper}) to these partition functions, since they satisfy the required conditions in that case.

\subsection{Hypergraph Extensions of the Potts and Random Cluster Model}
In this subsection, we introduce three hypergraph partition functions arising from natural hypergraph extensions of the Potts and random cluster models. 
As we will see, these partition functions are not equivalent in general, so we state them separately.
We begin with the partition function coming from Grimmett's many-body Potts model. 

\begin{definition}[Grimmett hypergraph Potts partition function \cite{grimmett1994potts}]\label{def:HyperPottsGrim}
      For a hypergraph $H$, the \emph{Grimmett hypergraph Potts partition function} is defined as
\[
\ZGr(H;q,w)
\coloneqq
\sum_{A\subseteq E(H)} q^{\kappa(A)}(w-1)^{|A|}.
\]
\end{definition}

The Potts model for hypergraphs has been studied to some extent, mainly under the labels \emph{many-body interactions} or \emph{higher-order interactions}.
Grimmett \cite{grimmett1994potts} describes the Potts model with many-body interactions in terms of a probability measure on the space of configurations, where the partition function is the normalizing factor. 
Here, the space of configurations is the space of all color assignments $V \to [q]$, for a positive integer $q$, and the weight of a configuration depends on a hyperedge interaction parameter $w$.
The resulting partition function (i.e.\ the normalizing factor), after applying a change of variables (up to an overall multiplicative factor, which does not affect the probability measure), is

\[
\sum_{\phi: V \to [q]}  \prod_{\substack{e \in E(H) \\ \phi(a)=\phi(b) \\\forall a,b \in e}} w .
\]
When $q$ is a positive integer, this is the natural spin representation. 
To extend the definition to arbitrary $q$, we pass to the subset expansion and take that as the definition. 
In this sense, the partition function of Definition~\ref{def:HyperPottsGrim} is implicit in Grimmett's formulation.

Evaluating $\ZGr$ for $w=v+1$  gives $\sum_{A \subseteq E} q^{\kappa(A)}v^{|A|}$, which is a hypergraph extension of the random cluster model.
The random cluster model with many-body interactions in \cite{grimmett1994potts} is being described using the same partition function as the partition function of the many-body interaction Potts model.

In Grimmett's version, each monochromatic edge has a contribution of $w$, independent of the degree of the edge. 
Motivated by the definition of the Tutte polynomial in Definition \ref{def: HyperTutte}, we incorporate the degrees of the hyperedge in the partition function of the Potts model and introduce the following hypergraph extension of the Potts partition function.

\begin{definition}[Degree-dependent hypergraph Potts partition function]\label{def:HyperPotts}
Let $H=(V,E)$ be a hypergraph and let $q,w \in \mathbb{C}$ be parameters. We define
\[
\ZP(H;q,w)
\;:=\;
\sum_{A\subseteq E} q^{\kappa(A)}\prod_{e\in A}\bigl(w^{d(e)-1}-1\bigr)
\]
to be the \emph{degree dependent hypergraph Potts partition function}.
For $q\in\mathbb{N}$ this polynomial has the equivalent spin-sum representation
\[
\ZP(H;q,w)
\;=\;
\sum_{\phi: V \to [q]} \prod_{\substack{e \in E(H) \\ \phi(a)=\phi(b) \\\forall a,b \in e}} w^{d(e) -1} .
\]
\end{definition}

Although we are not aware of a similar formulation of the Potts partition function in the literature, the idea of incorporating the  degrees in the `interaction energy' is not  novel. 
The Hamiltonians of the Ising model on hypergraphs, as defined in \cite{robiglio2025higher, son2024phase},  contribute a factor of $J_{|e|}$ to the total energy for each monochromatic edge $e$, where $|e|$ is the cardinality of $e$. 
Note however that, as we allow multiplicities in the edges, we distinguish the degree $d(e)$ of an edge $e$ from its `vertex size' $v(e)$.

Another natural hypergraph extension of the random cluster model is obtained by keeping the usual cluster-weight $q^{\kappa(A)}$ while allowing the edge-weight to depend on the degree, and choosing the contribution of each hyperedge $e\in A$ to be $v^{d(e)-1}$. This leads to the following definition.

\begin{definition}[Degree-dependent hypergraph random cluster partition function]\label{def: HyperRC}
For a hypergraph $H$, its degree-dependent random cluster partition function is given by
\[
Z_{\mathrm{RC}}(H;q,v) = \sum_{A \subseteq E} q^{ \kappa(A)} \prod_{e \in A}v^{d(e)-1}.
\]
\end{definition}

Unlike the graph case, $Z_P$ is not obtained from $Z_{\mathrm{RC}}$ anymore by a simple change of variables $v \mapsto w-1$.
In fact, what follows from the next subsection, is that $Z_P$ and $\ZRC$ are fundamentally different in the general hypergraph case.

\begin{remark}[Maximum entropy perspective]
The partition functions $\ZGr$, $\ZP$, and $\ZRC$ are all instances
of \emph{maximum entropy models}. 
In this framework, one seeks the probability distribution over spin
configurations that maximizes entropy subject to constraints on the
expected values of prescribed observables \cite{PhysRev.106.620}. 
In our setting the observables are local functions $\phi^e(\vec{s}\,)$
defined on the hyperedges of a hypergraph. 
This gives the probability distribution and partition function
\[
P(\vec{s}\,) = \frac{1}{Z}\,\prod_{e \in E}
\exp\!\bigl(g_e\,\phi^e(\vec{s}\,)\bigr),
\qquad
Z = \sum_{\vec{s}\,\in\,[q]^V}\,\prod_{e \in E}
\exp\!\bigl(g_e\,\phi^e(\vec{s}\,)\bigr),
\]
where the coupling constants $g_e$ are determined by requiring the
expectations
\[
\langle \phi^e(\vec{s}\,)\rangle_P
=
\sum_{\vec{s}} P(\vec{s}\,)\phi^e(\vec{s}\,)
\]
to match prescribed values. 
Maximum entropy models are widely used in inference, where the
couplings are fitted to empirical data so that the resulting
distribution captures observed correlations between variables
associated to the vertices and hyperedges of a network. 
In recent years, interest in maximum entropy models on hypergraphs has
surged in the context of statistical inference of higher-order moments
in data; see for instance \cite{battiston2021physics} for an overview of
higher-order interactions in statistical physics, and
\cite{spinmod, deClercq2026, de2025bayesian, Moody2026} for work on
maximum entropy models on hypergraphs.

The framework developed in \cite{Moody2026} makes this connection more
precise for models whose local interactions are determined by functions
on the spins of a hyperedge. 
For the Potts-type interactions considered here, the basic local
observable is
\[
\delta_e(\vec{s}\,) =
\begin{cases}
1, & \text{if all spins on $e$ agree,}\\
0, & \text{otherwise.}
\end{cases}
\]
For each of the three models, we may write
$g_e\phi^e(\vec{s}\,)=\lambda_e\delta_e(\vec{s}\,)$ for a suitable
edge-dependent parameter $\lambda_e$. Since $\delta_e(\vec{s}\,)$ only
takes the values $0$ and $1$, we have
\[
\exp\!\bigl(\lambda_e\delta_e(\vec{s}\,)\bigr)
=
1+\bigl(\exp(\lambda_e)-1\bigr)\delta_e(\vec{s}\,).
\]
Hence the partition function has the edge-subset expansion
\[
Z
=
\sum_{A\subseteq E} q^{\kappa(A)}
\prod_{e\in A} v_e
=
q^{v(H)}\widetilde Z(q,\mathbf v;r),
\]
where
\[
\widetilde Z(q,\mathbf v;r)
:=
\sum_{A\subseteq E}q^{-r(A)}\prod_{e\in A}v_e,
\qquad
r(A)=v(H)-\kappa(A),
\]
and $v_e=\exp(\lambda_e)-1$. Thus
$\widetilde Z(q,\mathbf v;r)$ is the multivariate rank generating
function, with one formal edge variable for each hyperedge. 
The partition functions $\ZGr$, $\ZP$, and $\ZRC$ considered in this
paper are obtained from it by univariate specializations of these edge
variables. 
If hyperedges contain repeated vertices, the agreement condition only
depends on the support of the hyperedge; the degree data is retained in
the specialized edge variables $v_e$.

For $\ZGr$ one takes $\lambda_e=\log w$, so that $v_e=w-1$; for $\ZP$
one takes $\lambda_e=(d(e)-1)\log w$, so that
$v_e=w^{d(e)-1}-1$; and for $\ZRC$ one takes
$\lambda_e=\log(1+v^{d(e)-1})$, so that $v_e=v^{d(e)-1}$.
Thus the three models have the same underlying monochromatic-edge
interaction, but encode the hyperedge degrees in different univariate
specializations of the multivariate rank generating function.

From this multivariate viewpoint, the uniform case is precisely the
case in which the degree-dependent edge variables become constant.
If $H$ is $(k+1)$-uniform, then $r(A)=v(H)-\kappa(A)$ is the rank
function of the associated $k$-polymatroid $P_H$, as used in
Section~\ref{sec: Hypergraphs}. 
Moreover, in each of the three specializations above the edge variable
$v_e$ is independent of $e$. 
For a uniform edge variable $u$, Proposition~5.4(2) of
\cite{Moody2026} gives
\[
\widetilde Z(q,u;r)
=
\left(\frac{u^{1/k}}{q}\right)^{r(E)}
T_k\!\left(P_H;1+\frac{q}{u^{1/k}},1+u^{1/k}\right).
\]
This is the rank-generating-function analogue of the specialization
$\THG(H;x,y)=T_k(P_H;x,y)$ for $(k+1)$-uniform hypergraphs, which is
developed in Section~\ref{sec: Hypergraphs}.
In Section~\ref{sec: potts}, and in particular in
Proposition~\ref{prop: RCTutteEquiv} and
Theorem~\ref{thm:incomparability}, we use these observations to compare
$\THG$, $\ZRC$, $\ZP$, and $\ZGr$. We see that $\ZRC$ is equivalent to
$\THG$, while outside the uniform setting the degree-dependent edge
variables separate the three hypergraph partition functions.
\end{remark}

\begin{theorem}[Deletion--contraction for partition functions]\label{thm:dc-ZP}
Let $H=(V,E)$ be a hypergraph and let $e\in E$.
Then the following deletion--contraction recurrences hold:
\begin{align}
\ZP (H;q,w)
&= \ZP (H\setminus e;q,w) + \bigl(w^{\,d(e)-1}-1\bigr)\,\ZP(H/e;q,w),\\
\ZGr(H;q,w)
&= \ZGr(H\setminus e;q,w) + (w-1)\,\ZGr(H/e;q,w),\\
\ZRC(H;q,v)
&= \ZRC(H\setminus e;q,v) + v^{\,d(e)-1}\,\ZRC(H/e;q,v),
\end{align}
with deletion and contraction as in Definition~\ref{def: del/con}.
\end{theorem}

\begin{proof}[Proof of Theorem~\ref{thm:dc-ZP}]
We will prove the statement for $\ZP$.
Using the same technique, we can verify the recurrence relations for $\ZRC$ and $\ZGr$.

Start from the subset expansion and split the sum into the cases $e\in A$ and $e \notin A$, giving
\begin{align*}
\ZP(H;q,w)
&=\sum_{\substack{A\subseteq E\\ e\notin A}} q^{\kappa(A)}\!\!\prod_{f\in A}\!\bigl(w^{\,d(f)-1}-1\bigr)
\;+\;\sum_{\substack{A\subseteq E\\ e\in A}} q^{\kappa(A)}\!\!\prod_{f\in A}\!\bigl(w^{\,d(f)-1}-1\bigr).
\end{align*}
The first sum is exactly $\ZP(H\setminus e;q,w)$.

For the second sum, set $A'=A\setminus\{e\}$; this gives a bijection
\[
\{A\subseteq E:\ e\in A\}\ \longleftrightarrow\ \{A'\subseteq E\setminus\{e\}\}.
\]
By the definition of contraction (Definition~\ref{def: del/con}), the degrees of all remaining hyperedges are unchanged, and the number of components behaves as
\[
\kappa\bigl( (H/e)\big|_{A'} \bigr)\;=\;\kappa\bigl( H\big|_{A'\cup\{e\}} \bigr).
\]
Hence, we can rewrite the second sum as
\begin{align*}
\sum_{\substack{A\subseteq E\\ e\in A}} q^{\kappa(A)}\!\!\prod_{f\in A}\!\bigl(w^{\,d(f)-1}-1\bigr)
&=\sum_{A'\subseteq E\setminus\{e\}} q^{\,\kappa((H/e)|_{A'})}
\Bigl(\prod_{f\in A'}\bigl(w^{\,d(f)-1}-1\bigr)\Bigr)\,\bigl(w^{\,d(e)-1}-1\bigr)\\
&=\bigl(w^{\,d(e)-1}-1\bigr)\,\ZP(H/e;q,w).
\end{align*}
Combining the two parts gives the stated recurrence.
\end{proof}

\subsection{Relating the Tutte Polynomial}

In this subsection, we relate the hypergraph Tutte polynomial $\THG$ to the partition functions $\ZGr$, $\ZP$, and $\ZRC$.
We first show that $\THG$ is equivalent to the degree-dependent hypergraph random cluster partition function $\ZRC$.
We then compare $\THG$, $\ZP$, and $\ZGr$ in distinguishing power (Definition~\ref{def: dist power}).

In particular, although $\ZP$ and $\ZRC$ both depend on edge degrees, only $\ZRC$ is equivalent to $\THG$ in our hypergraph setting.
By contrast, $\ZGr$ does not depend on edge degrees, and we show that it is incomparable with both $\THG$ and $\ZP$.
The main result of this subsection is Theorem~\ref{thm:incomparability}, which states that $\THG$, $\ZP$, and $\ZGr$ are pairwise incomparable in distinguishing power.
Using the equivalence of $\THG$ and $\ZRC$, we then deduce the corresponding statement for $\ZRC$, $\ZP$, and $\ZGr$, as well as the non-specialization corollary (Corollary~\ref{cor:NoSpecialization}).

\begin{proposition}\label{prop: RCTutteEquiv}
    Let $H$ be a   hypergraph. Then $T_{\mathrm{HG}}$ is realized as an evaluation of the  random cluster partition function $\ZRC$, and conversely $\ZRC$ is obtained as an evaluation of $T_{\mathrm{HG}}$, as follows:
    \begin{align}\label{eq:TasZ}
               \THG(H;x, y)  = (x-1)^{-\kappa(H)}(y-1)^{-v(H)}\, \ZRC\!\left(H; (x-1)(y-1),\, y-1\right), 
    \end{align}
    and
    \begin{align}\label{eq: ZasT}
               \ZRC(H;q,v)
        =
        q^{\kappa(H)}\,v^{\,v(H)-\kappa(H)}\;
        T_{\mathrm{HG}}\!\left(H;1+\frac{q}{v},\,1+v\right),
    \end{align}
    for $v \neq 0$.

\end{proposition}
\begin{proof}
    For Equation~\ref{eq:TasZ}, recall that the hypergraph Tutte polynomial can be rewritten as
    \[
     T_{\mathrm{HG}}(H;x+1,y+1)
    = (xy)^{-\kappa(H)}y^{-v(H) + \kappa(H)}\sum_{A\subseteq E}(xy)^{\,\kappa(A)}\,y^{\,d(A)-|A|}.
    \]
    Since 
    \[
    \ZRC(H;q,v)=\sum_{A\subseteq E} q^{\kappa(A)}\, v^{\,d(A)-|A|},
    \]
    we obtain
      \begin{align*}
     T_{\mathrm{HG}}(H;x+1, y+1)  
     &= (xy)^{-\kappa(H)}y^{-v(H) + \kappa(H)}\, \ZRC(H; xy, y)\\
     &= x^{-\kappa(H)}y^{-v(H)}\, \ZRC(H; xy, y).
      \end{align*}
    Equation~\ref{eq:TasZ}  follows by the change of variables $x\mapsto x-1$ and $y\mapsto y-1$. 
    Equation~\ref{eq: ZasT} then follows by rewriting Equation~\ref{eq:TasZ}.
\end{proof}

We compare hypergraph invariants via their distinguishing power (also called distinctive power), i.e., their ability to distinguish hypergraphs.
This approach appears in earlier work such as \cite{makowsky08}, and is explicitly defined in \cite{makowsky2025}.
\begin{definition}[Distinguishing power]\label{def: dist power}
Let $F_1,F_2$ be hypergraph invariants.

We say that $F_1$ is \emph{less distinguishing} than $F_2$ if whenever $F_2$ does not distinguish two hypergraphs, then neither does $F_1$, i.e.
\[
F_2(H)=F_2(H') \ \Longrightarrow\  F_1(H)=F_1(H') \qquad \text{for all } H,H'.
\]

If $F_1$ is less distinguishing than $F_2$ and $F_2$ is less distinguishing than $F_1$, then $F_1$ and $F_2$ are \emph{equally distinguishing}.
If neither $F_1$ is less distinguishing than $F_2$ nor $F_2$ is less distinguishing than $F_1$, then $F_1$ and $F_2$ are \emph{incomparable in distinguishing power}.
Equivalently, there exist $H_1,H_2$ with $F_1(H_1)=F_1(H_2)$ but $F_2(H_1)\neq F_2(H_2)$, and there exist $H_3,H_4$ with $F_2(H_3)=F_2(H_4)$ but $F_1(H_3)\neq F_1(H_4)$.
\end{definition}

\begin{theorem}\label{thm:incomparability}
The hypergraph polynomials $\THG$, $\ZP$ and $\ZGr$ are pairwise incomparable in distinguishing power, even among hypergraphs with the same global parameters, i.e. the same number of vertices, components, edges and total degree.

\end{theorem}

\begin{proof}
\smallskip
\noindent\emph{(i) Comparing $\THG$ and $\ZP$.}

Consider the hypergraphs $H_1$ and $H_2$ on the common vertex set
\[
V(H_1)=V(H_2)=\{a\}.
\]
Let
\[
E(H_1)=\{\{a,a\},\,\{a,a,a,a\}\},
\qquad
E(H_2)=\{\{a,a,a\},\,\{a,a,a\}\}.
\]
Then $v(H_1)=v(H_2)=1$, $\kappa(H_1)=\kappa(H_2)=1$, $e(H_1)=e(H_2)=2$, and
$d(H_1)=d(H_2)=6$.
A direct computation gives
\[
\THG(H_1;x,y)= y^4-3y^3+3y^2
\neq
y^4-4y^3+8y^2-8y+4
= \THG(H_2;x,y).
\]
On the other hand,
\[
\ZP(H_1;q,w)=\ZP(H_2;q,w)=q w^4.
\]
So $\THG$ is not less distinguishing than $\ZP$.

Conversely, consider the hypergraphs $H_3$ and $H_4$ on vertex set $\{a,b,c\}$ with
\[
E(H_3)=\bigl\{\{a,b\},\,\{a,b\},\,\{a,b,c\}\bigr\},
\qquad
E(H_4)=\bigl\{\{a,b\},\,\{a,c\},\,\{b,b,c\}\bigr\}.
\]
Again $v(H_3)=v(H_4)=3$, $e(H_3)=e(H_4)=3$, $\kappa(H_3)=\kappa(H_4)=1$, and
$d(H_3)=d(H_4)=7$.
A direct computation yields
\[
\THG(H_3;x,y)=\THG(H_4;x,y)=x^2+xy-x+y^2-y,
\]
whereas
\begin{align*}
\ZP(H_3;q,w)
&= q w^4 + q^2 w^2 - q w^2 + q^3 - q^2,\\
\ZP(H_4;q,w)
&= q w^4 + q^2 w^2 - q w^2 + 2q^2 w - 2q w + q^3 - 3q^2 + 2q,
\end{align*}
so $\ZP(H_3;q,w)\neq \ZP(H_4;q,w)$.
Thus $\ZP$ is also not less distinguishing than $\THG$.

\smallskip
\noindent\emph{(ii) Comparing $\THG$ and $\ZGr$.}

For $H_i$ as above, for $i=1,2,3,4$, we quickly verify that
$\ZGr(H_1;q,w)=\ZGr(H_2;q,w)$, and that $\ZGr(H_3;q,w)\neq \ZGr(H_4;q,w)$,
meaning that neither $\THG$ is less distinguishing than $\ZGr$ nor $\ZGr$ is less distinguishing than $\THG$.

\smallskip
\noindent\emph{(iii) Comparing $\ZP$ and $\ZGr$.}

Next, consider the hypergraphs $H_5$ and $H_6$ on the common vertex set
\[
V(H_5)=V(H_6)=\{a,b\}.
\]
Let
\[
E(H_5)=\bigl\{\{a,a\},\,\{a,a\},\,\{a,a,b\}\bigr\},
\qquad
E(H_6)=\bigl\{\{a,b\},\,\{a,b\},\,\{a,a,a\}\bigr\}.
\]
Then $v(H_5)=v(H_6)=2$, $\kappa(H_5)=\kappa(H_6)=1$, $e(H_5)=e(H_6)=3$, and
$d_2(H_5)=d_2(H_6)=2$ and $d_3(H_5)=d_3(H_6)=1$.
A direct computation gives
\[
\ZP(H_5;q,w)=\ZP(H_6;q,w)= q w^4 + q(q-1) w^2,
\]
whereas
\[
\ZGr(H_5;q,w)= q w^3 + q(q-1) w^2
\neq
q w^3 + q(q-1) w
= \ZGr(H_6;q,w).
\]
So $\ZGr$ is not less distinguishing than $\ZP$.

Conversely, consider the hypergraphs $H_7$ and $H_8$ on the common vertex set
\[
V(H_7)=V(H_8)=\{a,b\},
\]
with
\[
E(H_7)=\bigl\{\{a,a\},\,\{a,a,a\},\,\{a,b\}\bigr\},
\qquad
E(H_8)=\bigl\{\{a,a\},\,\{a,a\},\,\{a,a,b\}\bigr\}.
\]
A direct computation yields
\[
\ZP(H_7;q,w)= q w^4 + q(q-1) w^3
\neq
q w^4 + q(q-1) w^2
= \ZP(H_8;q,w),
\]
while
\[
\ZGr(H_7;q,w)=\ZGr(H_8;q,w)
= q w^3 + q(q-1) w^2.
\]
Thus $\ZP$ is also not less distinguishing than $\ZGr$.
\end{proof}

As $\ZRC$ and $\THG$ are equivalent, it follows that the hypergraph random cluster partition function $\ZRC$, $\ZP$ and $\ZGr$ are all pairwise incomparable in distinguishing power, even among hypergraphs with the same global parameters. As a consequence, we have the following corollary.

\begin{corollary}\label{cor:NoSpecialization}
The hypergraph partition functions $\ZRC$, $\ZP$ and $\ZGr$ do not specialize to each other, even up to multiplication by a factor depending only on global parameters.
\end{corollary}

When we restrict ourselves to the class of $(k+1)$-uniform hypergraphs however, both $\ZP$ and $\ZGr$ fall within the framework of Theorem~\ref{thm: univers-unif-hyper}, as is evident from Theorem~\ref{thm:dc-ZP}.
Applying Theorem~\ref{thm: univers-unif-hyper}, we obtain the following corollary.

\begin{corollary}
Let $k\in \mathbb{Z}_{\ge 1}$. For $H \in \mathcal{H}_{k+1}$, and $w^k \neq 1$, the following identity holds:
\[
\ZP(H;q,w) = q^{\kappa(H)}(w^{k}-1)^{\frac{v(H)-\kappa(H)}{k}}\THG\left(H; \frac{q}{(w^{k}-1)^{1/k}} +1 , (w^{k}-1)^{1/k} +1 \right).
\]
\end{corollary}

\begin{proof}
For an edgeless hypergraph $H$ on $n$ vertices, we have
\[
\ZP(H;q,w)=q^n.
\]
Moreover, $\ZP$ satisfies the deletion--contraction recurrence
\[
\ZP(H;q,w)=\ZP(H\setminus e;q,w) + (w^k-1)\,\ZP(H/e;q,w),
\]
for every hyperedge $e$ in a $(k+1)$-uniform hypergraph.
Hence $\ZP$ satisfies the hypotheses of Theorem~\ref{thm: univers-unif-hyper} with $\alpha=q$, and with $a,b,x,y$ chosen so that
\[
a^{k-\Delta}x^{\Delta} =1, \qquad b^{k-m}y^{m}=w^k-1,
\]
for every $\Delta,m \in \{0,\dots,k\}$.
Therefore $a=x$, $b=y$, $a^k=1$, and $b^k=w^k-1$.
Choosing $a=1$ and fixing a $k$-th root $b=(w^k-1)^{1/k}$ gives the result via \eqref{eq:univ-hyper}.
\end{proof}

Note that multiplying the choices of $a,b,x,y$ by a common $k$-th root of unity $\eta$ yields another valid choice.

Similarly, another application of Theorem~\ref{thm: univers-unif-hyper} gives, for $w \neq 1$,
\[
\ZGr(H;q,w)=q^{\kappa(H)}(w-1)^{\frac{v(H)-\kappa(H)}{k}}\THG\left(H;\frac{q}{(w-1)^{1/k}}+1,\,(w-1)^{1/k}+1\right),
\]
for a fixed choice of a $k$-th root of $w-1$.
Since the verification is entirely analogous, we omit the details.

\section{Comparing $\THG$ and $\mathcal{T}_P$} 
\label{sec: comparison}
In Section~\ref{sec: polymatroids} we have seen that the $k$-polymatroid Tutte polynomial is a specialization of $N(P,\omega;u,v)$ as in Equation~\ref{eq:chavez}, by choosing the right weights.  
The main organizational point is that, for $(k+1)$-uniform hypergraphs $H$, our hypergraph polynomial specializes to the $k$-polymatroid polynomial on the associated hypergraphical polymatroid $P_H$, namely $\THG(H;x,y)=T_k(P_H;x,y)$.  
Since the class of polymatroids arising from $(k+1)$-uniform hypergraphs is closed under deletion and contraction in our framework, this comparison is compatible with minors on both sides.

In this section, we compare $\THG$ with the polymatroid Tutte polynomial $\mathcal{T}_P$ of Bernardi, Kálmán, and Postnikov, as defined in \cite{bernardi2022universal}.  
There are two separate issues.
First, whether $\THG$ and $\mathcal{T}_P$ determine one another.
Second, whether the characteristic polynomial can be obtained as a specialization of $\mathcal{T}_P$.
Since incomparability of two invariants does not by itself rule out the possibility that a specialization of one is determined by the other, we treat these separately.

The polynomials $\THG$ and $\mathcal{T}_P$  differ both in their domain of definition and in the type of deletion--contraction theory they admit.
On one hand, we have $\mathcal{T}_P$, which was originally defined via bases and internal and external activities, and it is defined for all polymatroids. 
More recently, it has been shown that $\mathcal{T}_P$ also admits a polymatroid level recurrence \cite{guan2025recursive}. 
This recurrence has the form of a deletion--contraction formula, but it involves additional polymatroid slices.
Even when the input polymatroid is hypergraphical, these slices need not be hypergraphical, so from the hypergraph point of view the recursion typically leaves the hypergraphical setting \cite{bernardi2022universal}.

On the other hand, the hypergraph polynomial $\THG$ is defined for all hypergraphs and satisfies a deletion--contraction recurrence within the class of all hypergraphs, but it has no polymatroid extension. 
The polynomial $T_k$ sits between these two extremes. 
It admits a full deletion--contraction theory on its natural domain and it preserves the hypergraphical subclass arising from $(k+1)$-uniform hypergraphs, but it is not meant to cover all polymatroids and it does not apply to arbitrary non-uniform hypergraphs.

\subsection{The polymatroid Tutte polynomial $\mathcal{T}_P$}

We now give $\mathcal{T}_P$ via the recursive characterization of \cite{guan2025recursive}, which we will use as our working definition.

\begin{definition}[Recursive definition for $\mathcal{T}_P$]\label{def:TP-rec}
Let $P=(E,r)$ be a polymatroid. Define $\mathcal{T}_P(x,y)$ recursively on $|E|$ as follows.

If $E=\varnothing$, set $\mathcal{T}_P(x,y)=1$. Otherwise, fix $e\in E$ and set
\[
\alpha_e := r(E)-r(E\setminus\{e\}),
\qquad
\beta_e := r(\{e\}).
\]
For each integer $j$ with $\alpha_e\le j\le \beta_e$, let $P_{e,j}$ be the polymatroid on $E\setminus\{e\}$ with rank function
\[
r_{e,j}(A):=\min\{\,r(A),\, r(A\cup\{e\})-j\,\}\qquad(A\subseteq E\setminus\{e\}).
\]
Then $P_{e,\alpha_e}=P\setminus e$ and $P_{e,\beta_e}=P/e$, and we have
\[
\mathcal{T}_P(x,y)=
\begin{cases}
(x+y-1)\,\mathcal{T}_{P\setminus e}(x,y), & \alpha_e=\beta_e,\\[1mm]
x\,\mathcal{T}_{P\setminus e}(x,y)+y\,\mathcal{T}_{P/e}(x,y)+\displaystyle\sum_{j=\alpha_e+1}^{\beta_e-1} \mathcal{T}_{P_{e,j}}(x,y), & \alpha_e<\beta_e.
\end{cases}
\]
\end{definition}

\begin{remark}
It is shown in \cite{guan2025recursive} that this recursion is well defined, in particular independent of the choice of $e$, and agrees with the Bernardi--K\'alm\'an--Postnikov polymatroid Tutte polynomial from \cite{bernardi2022universal}.

In the matroid case, $\beta_e-\alpha_e\le 1$, so the slice sum is empty and the recursion has only deletion and contraction terms.
Moreover, if $P$ is a matroid, $\mathcal{T}_{P}$ does not reduce to the classical matroid Tutte polynomial $T_M$, but is related to it by a simple change of variables:
\[
\mathcal{T}_{P(M)}(x,y)
=
x^{\,n-d}y^{\,d}\,
T_M\!\left(\frac{x+y-1}{y},\,\frac{x+y-1}{x}\right)
\qquad
\text{\cite[Thm.~5.2]{bernardi2022universal}},
\]
where $n$ is the size of the matroid and $d$ the total rank.
\end{remark}

\subsection{Comparison with $\THG$}

We now compare $\THG$ and $\mathcal{T}_P$ on explicit pairs of $4$-uniform hypergraphs.
The polynomial equalities and inequalities below can be verified using the code available at \url{https://github.com/KhallilMaliki/HyperTutteExtensions}.
We write $\mathcal{T}_P(H)$ as shorthand for $\mathcal{T}_P(P_H)$.

\begin{proposition}\label{prop:thg-tp-incomparable}
The invariants $\THG$ and $\mathcal{T}_P$ are incomparable in distinguishing power, even among $4$-uniform hypergraphs with identical global parameters.
\end{proposition}

\begin{proof}
Consider first the pair $H_1,H_2$ of $4$-uniform hypergraphs on $V=\{a,b,c,d,e,f,g\}$ with
\[
E(H_1)=\bigl\{\{a,b,c,d\},\{a,b,c,e\},\{a,b,d,f\},\{a,c,f,g\}\bigr\},
\]
\[
E(H_2)=\bigl\{\{a,b,c,d\},\{a,b,c,e\},\{a,b,f,g\},\{a,c,f,g\}\bigr\}.
\]
See Figure~\ref{fig:h1h2} for the associated bipartite graphs. 

A computation gives
\[
\THG(H_1)=\THG(H_2),
\qquad
\mathcal{T}_P(H_1)\neq \mathcal{T}_P(H_2).
\]
The differing coefficients of $\mathcal{T}_P(H_1)$ and $\mathcal{T}_P(H_2)$ are listed in Table~\ref{tab:pair1-tp}.
Hence $\THG$ does not determine $\mathcal{T}_P$.

Now consider the pair $H_3,H_4$ of $4$-uniform hypergraphs on $V=\{a,b,c,d,e,f,g,h\}$ with
\[
E(H_3)=\bigl\{\{a,b,c,h\},\{a,b,d,h\},\{a,b,e,h\},\{a,c,f,h\},\{a,d,g,h\}\bigr\},
\]
\[
E(H_4)=\bigl\{\{a,b,c,h\},\{a,b,d,h\},\{a,b,e,h\},\{a,c,f,h\},\{a,f,g,h\}\bigr\}.
\]
See Figure~\ref{fig:h3h4} for the associated bipartite graphs.

A computation gives
\[
\mathcal{T}_P(H_3)=\mathcal{T}_P(H_4),
\qquad
\THG(H_3)\neq \THG(H_4).
\]
The differing coefficients of $\THG(H_3)$ and $\THG(H_4)$ are listed in Table~\ref{tab:pair2-thg}.
Hence $\mathcal{T}_P$ does not determine $\THG$.

Therefore $\THG$ and $\mathcal{T}_P$ are incomparable in distinguishing power.
\end{proof}

For later use, we record the full $\THG$ polynomials for the second pair:
\[
\begin{aligned}
\THG(H_3;x,y)={}&x^7-7x^6+21x^5-30x^4+5x^3y^2-10x^3y+20x^3 \\
&+5x^2y^4-20x^2y^3+15x^2y^2+15x^2y-6x^2 \\
&+3xy^6-18xy^5+35xy^4-16xy^3-12xy^2-6xy+x \\
&+y^8-8y^7+25y^6-36y^5+20y^4+y^2+y
\end{aligned}
\]
\[
\begin{aligned}
\THG(H_4;x,y)={}&x^7-7x^6+21x^5-30x^4+5x^3y^2-10x^3y+20x^3 \\
&+4x^2y^4-16x^2y^3+9x^2y^2+19x^2y-7x^2 \\
&+3xy^6-18xy^5+37xy^4-22xy^3-6xy^2-8xy+x \\
&+y^8-8y^7+25y^6-36y^5+19y^4+2y^3+y
\end{aligned}
\]

\begin{table}
\centering
\tbl{\caption{Monomials whose coefficients differ in $\mathcal{T}_P(H_1)$ and $\mathcal{T}_P(H_2)$.}\label{tab:pair1-tp}}
{\begin{tabular}{@{}lcc@{}}
\hline
Monomial & $\mathcal{T}_P(H_1)$ & $\mathcal{T}_P(H_2)$ \\
\hline
$xy$  & $1$  & $0$  \\
$x$   & $-3$ & $-4$ \\
$y^2$ & $0$  & $-1$ \\
$1$   & $-1$ & $0$  \\
\hline
\end{tabular}}
{}
\end{table}

\begin{table}
\centering
\tbl{\caption{Monomials whose coefficients differ in $\THG(H_3)$ and $\THG(H_4)$.\label{tab:pair2-thg}}}
{\begin{tabular}{@{}lcc@{}}
\hline
Monomial & $\THG(H_3)$ & $\THG(H_4)$ \\
\hline
$x^2y^4$ & $5$   & $4$   \\
$x^2y^3$ & $-20$ & $-16$ \\
$x^2y^2$ & $15$  & $9$   \\
$x^2y$   & $15$  & $19$  \\
$x^2$    & $-6$  & $-7$  \\
$xy^4$   & $35$  & $37$  \\
$xy^3$   & $-16$ & $-22$ \\
$xy^2$   & $-12$ & $-6$  \\
$xy$     & $-6$  & $-8$  \\
$y^4$    & $20$  & $19$  \\
$y^3$    & $0$   & $2$   \\
$y^2$    & $1$   & $0$   \\
\hline
\end{tabular}}
{}
\end{table}
\begin{corollary}\label{cor:char-not-specialization}
The characteristic polynomial is not, in general, a specialization of $\mathcal{T}_P$.
\end{corollary}

\begin{proof}
For the hypergraphs $H_3$ and $H_4$ above, we have
\[
\mathcal{T}_P(H_3)=\mathcal{T}_P(H_4).
\]
Since $H_3$ and $H_4$ are $4$-uniform, we have $k=3$, and by Proposition~\ref{prop: char-tutte},
\[
\chi(H;x)=(-1)^{r(P_H)}\THG(H;1-x,0).
\]
Computing $\chi$ from the displayed formulas for $\THG(H_3;x,y)$ and $\THG(H_4;x,y)$, we obtain
\[
\chi(H_3;x)=x^7-5x^4+5x^3-x,
\qquad
\chi(H_4;x)=x^7-5x^4+5x^3+x^2-3x+1.
\]
Hence
\[
\chi(H_3;x)\neq \chi(H_4;x).
\]
Therefore the characteristic polynomial is not determined by $\mathcal{T}_P$.
In particular, it cannot be obtained as a specialization of $\mathcal{T}_P$.
\end{proof}

Since, in the setting of hypergraphs viewed through their associated polymatroids, the chromatic polynomial considered in \cite{bergehyper} differs from the characteristic polynomial only by a simple factor (see \cite{helgason1972aspects, whittle1992geometric}), it is not determined by $\mathcal{T}_P$ either.
This gives a negative answer to the question raised by Bernardi, Kálmán, and Postnikov in Section~18.3 of \cite{bernardi2022universal}.

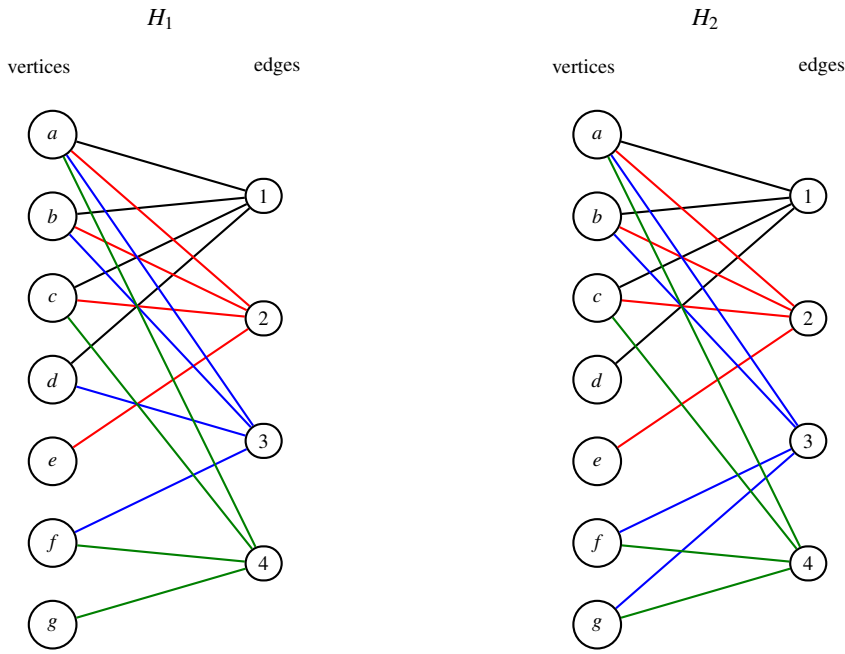
\begin{figure}
\centering
\begin{tikzpicture}[
    x=1cm,y=1cm,
    scale=0.9,
    transform shape,
    vtx/.style={circle, draw, thick, minimum size=7mm, inner sep=0pt, font=\small},
    hed/.style={rounded rectangle, draw, thick, minimum width=6mm, minimum height=5mm, inner sep=1pt, font=\small},
    lab/.style={font=\small},
    title/.style={font=\normalsize\bfseries}
]

\node[title] at (1.55,5.3) {$H_1$};

\node[lab] at (-0.2,4.6) {vertices};
\node[lab] at (3.3,4.6) {edges};

\node[vtx] (a1) at (0,3.6) {$a$};
\node[vtx] (b1) at (0,2.4) {$b$};
\node[vtx] (c1) at (0,1.2) {$c$};
\node[vtx] (d1) at (0,0.0) {$d$};
\node[vtx] (e1) at (0,-1.2) {$e$};
\node[vtx] (f1) at (0,-2.4) {$f$};
\node[vtx] (g1) at (0,-3.6) {$g$};

\node[hed] (h11) at (3.1, 2.7) {$1$};
\node[hed] (h12) at (3.1, 0.9) {$2$};
\node[hed] (h13) at (3.1,-0.9) {$3$};
\node[hed] (h14) at (3.1,-2.7) {$4$};

\draw[thick]               (a1) -- (h11);
\draw[thick]               (b1) -- (h11);
\draw[thick]               (c1) -- (h11);
\draw[thick]               (d1) -- (h11);

\draw[thick,red]           (a1) -- (h12);
\draw[thick,red]           (b1) -- (h12);
\draw[thick,red]           (c1) -- (h12);
\draw[thick,red]           (e1) -- (h12);

\draw[thick,blue]          (a1) -- (h13);
\draw[thick,blue]          (b1) -- (h13);
\draw[thick,blue]          (d1) -- (h13);
\draw[thick,blue]          (f1) -- (h13);

\draw[thick,green!50!black](a1) -- (h14);
\draw[thick,green!50!black](c1) -- (h14);
\draw[thick,green!50!black](f1) -- (h14);
\draw[thick,green!50!black](g1) -- (h14);

\node[title] at (9.55,5.3) {$H_2$};

\node[lab] at (7.8,4.6) {vertices};
\node[lab] at (11.3,4.6) {edges};

\node[vtx] (a2) at (8,3.6) {$a$};
\node[vtx] (b2) at (8,2.4) {$b$};
\node[vtx] (c2) at (8,1.2) {$c$};
\node[vtx] (d2) at (8,0.0) {$d$};
\node[vtx] (e2) at (8,-1.2) {$e$};
\node[vtx] (f2) at (8,-2.4) {$f$};
\node[vtx] (g2) at (8,-3.6) {$g$};

\node[hed] (h21) at (11.1, 2.7) {$1$};
\node[hed] (h22) at (11.1, 0.9) {$2$};
\node[hed] (h23) at (11.1,-0.9) {$3$};
\node[hed] (h24) at (11.1,-2.7) {$4$};

\draw[thick]               (a2) -- (h21);
\draw[thick]               (b2) -- (h21);
\draw[thick]               (c2) -- (h21);
\draw[thick]               (d2) -- (h21);

\draw[thick,red]           (a2) -- (h22);
\draw[thick,red]           (b2) -- (h22);
\draw[thick,red]           (c2) -- (h22);
\draw[thick,red]           (e2) -- (h22);

\draw[thick,blue]          (a2) -- (h23);
\draw[thick,blue]          (b2) -- (h23);
\draw[thick,blue]          (f2) -- (h23);
\draw[thick,blue]          (g2) -- (h23);

\draw[thick,green!50!black](a2) -- (h24);
\draw[thick,green!50!black](c2) -- (h24);
\draw[thick,green!50!black](f2) -- (h24);
\draw[thick,green!50!black](g2) -- (h24);

\end{tikzpicture}
\caption{Associated bipartite graphs of $H_1$ and $H_2$.}
\label{fig:h1h2}
\end{figure}

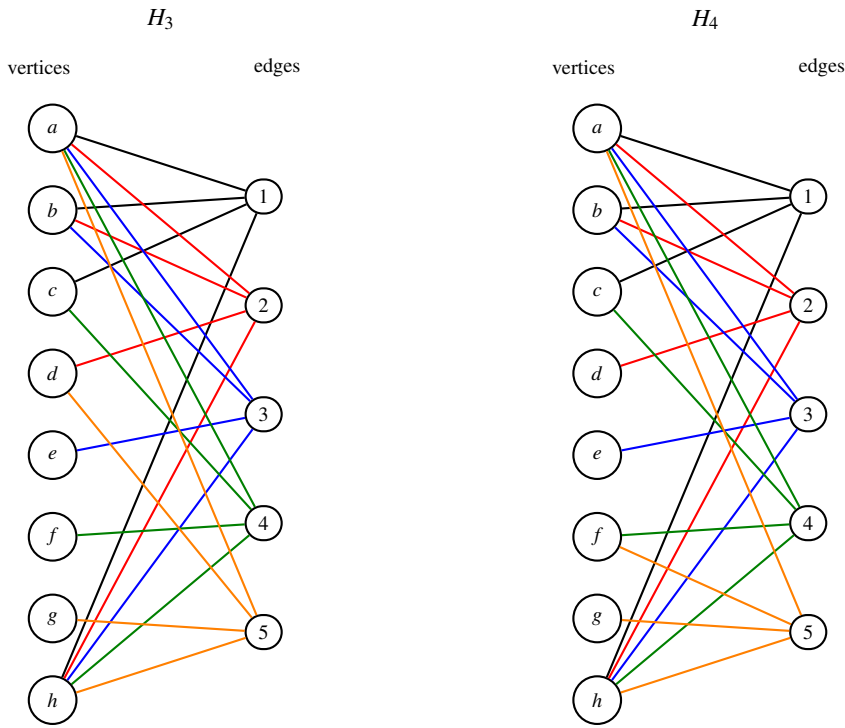
\begin{figure}
\centering
\begin{tikzpicture}[
    x=1cm,y=1cm,
    scale=0.9,
    transform shape,
    vtx/.style={circle, draw, thick, minimum size=7mm, inner sep=0pt, font=\small},
    hed/.style={rounded rectangle, draw, thick, minimum width=6mm, minimum height=5mm, inner sep=1pt, font=\small},
    lab/.style={font=\small},
    title/.style={font=\normalsize\bfseries}
]

\node[title] at (1.55,5.8) {$H_3$};

\node[lab] at (-0.2,5.1) {vertices};
\node[lab] at (3.3,5.1) {edges};

\node[vtx] (a3) at (0,4.2) {$a$};
\node[vtx] (b3) at (0,3.0) {$b$};
\node[vtx] (c3) at (0,1.8) {$c$};
\node[vtx] (d3) at (0,0.6) {$d$};
\node[vtx] (e3) at (0,-0.6) {$e$};
\node[vtx] (f3) at (0,-1.8) {$f$};
\node[vtx] (g3) at (0,-3.0) {$g$};
\node[vtx] (h3) at (0,-4.2) {$h$};

\node[hed] (h31) at (3.1, 3.2) {$1$};
\node[hed] (h32) at (3.1, 1.6) {$2$};
\node[hed] (h33) at (3.1, 0.0) {$3$};
\node[hed] (h34) at (3.1,-1.6) {$4$};
\node[hed] (h35) at (3.1,-3.2) {$5$};

\draw[thick]                 (a3) -- (h31);
\draw[thick]                 (b3) -- (h31);
\draw[thick]                 (c3) -- (h31);
\draw[thick]                 (h3) -- (h31);

\draw[thick,red]             (a3) -- (h32);
\draw[thick,red]             (b3) -- (h32);
\draw[thick,red]             (d3) -- (h32);
\draw[thick,red]             (h3) -- (h32);

\draw[thick,blue]            (a3) -- (h33);
\draw[thick,blue]            (b3) -- (h33);
\draw[thick,blue]            (e3) -- (h33);
\draw[thick,blue]            (h3) -- (h33);

\draw[thick,green!50!black]  (a3) -- (h34);
\draw[thick,green!50!black]  (c3) -- (h34);
\draw[thick,green!50!black]  (f3) -- (h34);
\draw[thick,green!50!black]  (h3) -- (h34);

\draw[thick,orange]          (a3) -- (h35);
\draw[thick,orange]          (d3) -- (h35);
\draw[thick,orange]          (g3) -- (h35);
\draw[thick,orange]          (h3) -- (h35);

\node[title] at (9.55,5.8) {$H_4$};

\node[lab] at (7.8,5.1) {vertices};
\node[lab] at (11.3,5.1) {edges};

\node[vtx] (a4) at (8,4.2) {$a$};
\node[vtx] (b4) at (8,3.0) {$b$};
\node[vtx] (c4) at (8,1.8) {$c$};
\node[vtx] (d4) at (8,0.6) {$d$};
\node[vtx] (e4) at (8,-0.6) {$e$};
\node[vtx] (f4) at (8,-1.8) {$f$};
\node[vtx] (g4) at (8,-3.0) {$g$};
\node[vtx] (h4) at (8,-4.2) {$h$};

\node[hed] (h41) at (11.1, 3.2) {$1$};
\node[hed] (h42) at (11.1, 1.6) {$2$};
\node[hed] (h43) at (11.1, 0.0) {$3$};
\node[hed] (h44) at (11.1,-1.6) {$4$};
\node[hed] (h45) at (11.1,-3.2) {$5$};

\draw[thick]                 (a4) -- (h41);
\draw[thick]                 (b4) -- (h41);
\draw[thick]                 (c4) -- (h41);
\draw[thick]                 (h4) -- (h41);

\draw[thick,red]             (a4) -- (h42);
\draw[thick,red]             (b4) -- (h42);
\draw[thick,red]             (d4) -- (h42);
\draw[thick,red]             (h4) -- (h42);

\draw[thick,blue]            (a4) -- (h43);
\draw[thick,blue]            (b4) -- (h43);
\draw[thick,blue]            (e4) -- (h43);
\draw[thick,blue]            (h4) -- (h43);

\draw[thick,green!50!black]  (a4) -- (h44);
\draw[thick,green!50!black]  (c4) -- (h44);
\draw[thick,green!50!black]  (f4) -- (h44);
\draw[thick,green!50!black]  (h4) -- (h44);

\draw[thick,orange]          (a4) -- (h45);
\draw[thick,orange]          (f4) -- (h45);
\draw[thick,orange]          (g4) -- (h45);
\draw[thick,orange]          (h4) -- (h45);

\end{tikzpicture}
\caption{Associated bipartite graphs of $H_3$ and $H_4$.}
\label{fig:h3h4}
\end{figure}

\section{Concluding Remarks} \label{sec: conclusion}
In this paper, we introduced the hypergraph Tutte polynomial $\THG$ and the $k$-polymatroid Tutte polynomial $T_k$, and established several Tutte type properties for these invariants.
In particular, $\THG$ provides a deletion--contraction theory on the class of all hypergraphs, while $T_k$ yields a corresponding Tutte polynomial theory for $k$-polymatroids.
For $(k+1)$-uniform hypergraphs, these two viewpoints meet through the associated polymatroid construction, in the sense that $\THG(H;x,y)=T_k(P_H;x,y)$, which lets us transport the universality theorem and convolution identities for $T_k$ back to the hypergraph setting.

We also placed $\THG$ in the broader hypergraph partition function framework.
In particular, we showed that $\THG$ is equivalent to the degree-dependent random cluster partition function $\ZRC$, thereby placing $\THG$ naturally in the hypergraph Potts/random cluster framework.
At the same time, our comparison results show that, outside the uniform setting, the hypergraph partition functions $\ZRC$, $\ZP$, and $Z_{\mathrm{Grimmett}}$ are genuinely different.

We also compared $\THG$ with the polymatroid Tutte polynomial $\mathcal{T}_P$.
More specifically, we showed that $\THG$ and the polymatroid Tutte polynomial $\mathcal{T}_P$ are incomparable in distinguishing power, even on uniform examples.
In particular, together with our identification of the characteristic polynomial with $T_k$ in the uniform setting (Proposition~\ref{prop: char-tutte}), this gives a negative answer to the question raised by Bernardi, Kálmán, and Postnikov in \cite{bernardi2022universal} of whether the characteristic polynomial arises as a specialization of $\mathcal{T}_P$.
Since, in this setting, the chromatic polynomial differs from the characteristic polynomial only by a simple factor, the same conclusion applies to the chromatic polynomial as well.

Several natural questions remain open.
A first question is whether $\THG$ admits an activity expansion, or more generally a direct combinatorial interpretation of its coefficients in terms of suitable hypergraph structures, at least in the uniform case, in analogy with the existing activity based expansions of the polymatroid Tutte polynomial $\mathcal{T}_P$ in \cite{bernardi2022universal, kalman2013version}.
In particular, unlike for $\mathcal{T}_P$, whose coefficients count directly interpretable objects in terms of basis elements, spanning trees, and activities, there is not yet a comparable interpretation for the coefficients of $\THG$.
As the example of the surface Tutte polynomial illustrates, such an expansion may already be quite subtle \cite{thompson2026quasi}.

Another direction is to extend the present theory to hypermaps.
Our definition of $\THG$ was partly inspired by the coarse Tutte polynomial for hypermaps introduced in \cite{ellis2024coarse}, and it is natural to ask whether the universality theory developed here can also be extended to the hypermap setting in the uniform case.

A further direction is to study special evaluations of $\THG$ and their computational complexity.
Since, for planar hypergraphs, $\THG$ agrees with the coarse hypermap Tutte polynomial from \cite{ellis2024coarse}, we can translate complexity questions from that setting into ours.
In particular, we can ask about the complexity of the following decision problem: given a planar hypergraph $H$, determine whether $\THG(H;1,1)=0$ or $\THG(H;1,1)>0$.

Finally, it would be worthwhile to understand more precisely the extent to which the uniform case is exceptional.
In the present paper, the uniform setting is exactly where the hypergraph and polymatroid theories align cleanly, and where the various partition functions become equivalent after a suitable change of variables.
It would be interesting to determine whether there are broader classes of non-uniform hypergraphs for which some form of universality theorem, convolution identity, or equivalence with other partition functions still survives.

\vspace{-2pt}

\section*{Acknowledgements}
We thank Guus Regts and Clelia Mulatier for their valuable input, insights, and constructive conversations.

\vspace{-2pt}

\bibliographystyle{plain}
\bibliography{references_axv}
	
\label{lastpage}

\end{document}